\documentclass[11pt]{article}

\usepackage{geometry}
\usepackage{graphicx}
\usepackage{dsfont} %% Enables \mat
\usepackage{amsmath}
\usepackage{amssymb}
\usepackage{amsthm}
\usepackage{url}

\usepackage{color}
\usepackage{subcaption}

\newcommand{\mc}{\mathcal}
\newcommand{\doi}[1]{\href{http://dx.doi.org/#1}{\normalsize{\textsc{doi:}}~\nolinkurl{#1}}}
\newcommand{\arxiv}[1]{\href{http://arxiv.org/abs/#1}{\normalsize{\textsc{arxiv:}}~\nolinkurl{#1}}}

\let\oldphi\phi

\renewcommand{\epsilon}{\varepsilon}
\renewcommand{\phi}{\varphi}

\newcommand{\R}{\mathbb{R}}

\newcommand{\N}{\mathbb{N}}

\renewcommand{\implies}{\;\Rightarrow\;}

\let\originalleft\left
\let\originalright\right
\renewcommand{\left}{\mathopen{}\mathclose\bgroup\originalleft}
\renewcommand{\right}{\aftergroup\egroup\originalright}

\def\clap#1{\hbox to 0pt{\hss#1\hss}}

\newcommand{\norm}[1]{\left\lVert #1\right\rVert}

\newcommand{\set}[1]{\left\{ #1\right\}}

\let\p=\paren

\newcommand{\sqparen}[1]{\left[ #1\right]}

\let\sp=\sqparen

\newcommand{\derivoper}{\mathrm{d}}

\newcommand{\deriv}[3][]{%
  \ifx\relax#1\relax{% %% If #1 is empty
    \frac{\derivoper #2}{\derivoper #3}%
  }\else{%
    \frac{\derivoper^{#1} #2}{\derivoper #3^{#1}}%
  }\fi%
}
\newcommand{\pderiv}[3][]{%
  \ifx\relax#1\relax{% %% If #1 is empty
    \frac{\partial #2}{\partial #3}%
  }\else{%
    \frac{\partial^{#1} #2}{\partial #3^{#1}}%
  }\fi%
}
\newcommand{\grad}[2][]{%
  \ifx\relax#1\relax{% %% If #1 is empty
    \nabla #2%
  }\else{%
    \nabla_{#1} #2%
  }\fi%
}
\newcommand{\diver}[2][]{%
  \ifx\relax#1\relax{% %% If #1 is empty
    \nabla \cdot #2%
  }\else{%
    \nabla_{#1} \cdot #2%
  }\fi%
}
\newcommand{\curl}[2][]{%
  \ifx\relax#1\relax{% %% If #1 is empty
    \nabla \times #2%
  }\else{%
    \nabla_{#1} \times #2%
  }\fi%
}
\newcommand{\lapl}[2][]{%
  \ifx\relax#1\relax{% %% If #1 is empty
    \Delta #2%
  }\else{%
    \Delta_{#1} #2%
  }\fi%
}

\newcommand{\smat}[1]{%
  \begin{bmatrix}#1 \end{bmatrix}%
}
\newcommand{\smats}[1]{%
  \sqparen{\begin{smallmatrix}#1 \end{smallmatrix}}%
}

\usepackage[utf8]{inputenc}
\usepackage[T1]{fontenc}%\usepackage{clomment}
\usepackage[T1]{fontenc}
\usepackage{cleveref}

\newtheorem{assumption}{Assumption}
\newtheorem{definition}{Definition}
\newtheorem{lemma}{Lemma}
\newtheorem{theorem}{Theorem}
\newtheorem{proposition}{Proposition}
\usepackage{comment}

\usepackage{mathrsfs}

\DeclareMathAlphabet{\mathpzc}{OT1}{pzc}{m}{it}

\begin{document}
\title{%
Technichal Report: Optimal Control of Piecewise-Smooth Control Systems  via Singular Perturbations}
\author{$\text{Tyler Westenbroek}$ \thanks{ Dept of EECS., University of California, Berkeley, westenbroekt@berkeley.edu. Corresponding Author.
        }  \and $\text{Xiaobin Xiong}$ \thanks{ Dept. of Mechanical and Civil Engineering, California Institute of Technology, 
        xxiong@caltech.edu
        }%
        % <-this % stops a spaces
         \and $\text{Aaron D. Ames}$ \thanks{ Dept. of Mechanical and Civil Engineering, California Institute of Technology, 
        ames@caltech.edu
        }%
        % <-this % stops a spaces
 \and $\text{S. Shankar Sastry}$ \thanks{ Dept of EECS., University of California, Berkeley, sastry@berkeley.edu}
        } % <-this % stops a space

\maketitle
\begin{abstract}
This paper investigates optimal control problems formulated over a class of piecewise-smooth vector fields. Instead of optimizing over the discontinuous system directly, we instead formulate optimal control problems over a family of regularizations which are obtained by "smoothing out" the discontinuity in the original system. It is shown that the smooth problems can be used to obtain accurate derivative information about the non-smooth problem, under standard regularity conditions. We then indicate how the regularizations can be used to consistently approximate the non-smooth optimal control problem in the sense of Polak. The utility of these smoothing techniques is demonstrated in an in-depth example, where we utilize recently developed reduced-order modeling techniques from the dynamic walking community to generate motion plans across contact sequences for a 18-DOF model of a lower-body exoskeleton. 
\end{abstract}

%%%%%%%%%%%%%%%%%%%%%%
%%%%%%% Introduction %%%%%%%%
%%%%%%%%%%%%%%%%%%%%%%

\section{Introduction}
%\cite{Supplementary}
Non-smooth dynamical systems naturally arise when modelling a vast array of engineering systems, and are familiar to researchers in areas ranging from the dynamic walking community \cite{xiong2018coupling}, \cite{westervelt2003hybrid} to the domain of power systems analysis \cite{hiskens2000trajectory}. However, despite the pervasive nature of discontinuous dynamics in systems theory and control, optimizing system trajectories through unplanned sequences of discontinuities remains a distinct technical challenge.

In particular, non-smooth systems are well known to display non-differentiability with respect to initial conditions and inputs \cite[Chapter 2, Section 11]{filippov2013differential}. This fundamental departure from smooth control theory has meant that the powerful optimization-based techniques designed for smooth control systems \cite{polak2012optimization} cannot be readily applied to the non-smooth setting. As an example, in the context of optimal control, this has traditionally led to the implementation of restrictive practices, such as forming large mixed-integer programs to reason over all possible "mode" sequences the trajectories may undergo \cite{borrelli2017predictive}, or limiting the class of system behaviors by fixed the number and sequence of discontinuities the trajectory of the system encounters \emph{a priori} \cite{kamgarpour2012optimal}, \cite{shaikh2007hybrid}.

\subsection{Contributions}
In this work, we formulate an optimal optimal control problem over a class of piecewise-smooth vector fields. Rather than reasoning through the discontinuities of the system directly, we instead study a family of regularizations which are obtained by "smoothing" the original non-smooth system  along the discontinuity. We then formulate approximate optimal control problems over these smooth control systems, which can be solved using standard derivative-based techniques. Theoretical guarantees are discussed, demonstrating when the regularized problem can be used to obtain accurate derivative information about the non-smooth problem. We then discuss how this result can be used to derive consistent approximations in the sense of Polak \cite{polak2012optimization}.

Finally, we use the smoothing approach to generate trajectories for a model of a bipedal robot, specifically a powered exoskeleton, which intermittently makes contact with the ground. In particular, we build on \cite{xiong2018bipedal}, \cite{xiong2019exo}, which introduced a class of embedded reduced-order modeling techniques to simplify the complexity of the motion planning problem. We believe the unification of these techniques represents an important new direction for robotic trajectory optimization through contact.
\subsection{Relation to the Literature}
The roots of the smoothing technique considered in this paper can be traced back to the literature on \emph{geometric singular perturbation theory} (see e.g. \cite{sastry2013nonlinear} or \cite{khalil1996noninear}). The  precise smoothing approach considered in this paper were first considered for autonomous piecewiese-smooth systems in work pioneered by Teixeira  \cite{llibre2008sliding}, \cite{llibre1997regularization}.  Emerging applications of these smoothing techniques in the control community can be found in  \cite{fiore2017observer}, \cite{fiore2016contraction}. The authors have also recently been made aware of a contribution \cite{stewart2010optimal} which appears to have been made independently of the above literature and in which similar smoothing techniques were also considered in the context of optimal control. Our analysis builds on the results of \cite{stewart2010optimal} by requiring fewer technical assumptions, proving a concrete rate of convergence in Theorem \ref{thm:convergence1}, and formalizing the connection between the minimizers of the regularized and discontinuous problems in Section \ref{sub:minimizers}. Finally we mention a recent contribution \cite{westenbroek2018new}, which builds on the geometric theory of hybrid systems \cite{simic2005towards},\cite{burden2015metrization},\cite{ames2005homology}, in which the smoothing techniques discussed in this paper were extended to a class of hybrid dynamical systems where the trajectories of the system undergo autonomous "jumps". We hope to extend the results presented here to this more general setting in a forthcoming article.

\label{sec:intro}

\subsection{Notation}
\label{sec:math}
We now fix notation used throughout the document. Unless otherwise noted, the 2-norm is our choice of norm for finite dimensional spaces, and the $L^2$ norm is our norm of choice for function spaces. We say that a function $f \colon \R^n \to \R^p$ is \emph{Lipschitz continuous} if there exists $L >0$ such that for each $x_1,x_2 \in R^n$ we have $||f(x_1)- f(x_2)|| \leq L ||x_1 - x_2||$.  Given a smooth function $h \colon \R^n \to \R$ and a controlled vector field $f \colon \R^n \times \R^m \to \R^n$, the \emph{Lie Derivative} of $h$ along $f$ is denoted $L_f h(x,u) \colon = \deriv{}{x}h(x) \cdot f(x,u)$. Given a normed space $V$, $\mc{P}(V)$ denotes the set of all subsets of $V$. Given a subset $S \subset V$, $\overline{co}S$ denotes the convex closure of $S$ in V. 

For compactness of notation and exposition, throughout the paper all of the control systems we consider will have states belonging to $\R^n$ and admissible inputs taking values in $\R^m$, where $n$ and $m$ are positive integers. Moreover, we will assume that the admissible set of initial conditions for these control systems belong to a set $D  \subset \R^n$, and that the control supplied to these systems takes on values in the set $U \subset \R^m$. Both $D$ and $U$ are assumed to be compact, connected and convex. Throughout the paper we will consider square-integrable control signals of length $T>0$, which will be denoted $L^2([0,T],U)$. We will compactly denote the admissible data for these systems by
\begin{equation*}
\mc{X} = D \times L^2(\sp{0,T}, U).
\end{equation*}
We define the space of tangent vectors on $\mc{X}$ by
\begin{equation*}
\mc{X}' = \R^n \times L^2(\sp{0,T}, R^m).
\end{equation*}
Finally, we endow both $\mc{X}$ and $\mc{X}'$ with the norm ${|| \cdot || \colon \R^n \times L^2(\sp{0,T}, R^m) \to \R}$ defined for each $\xi = (x_0, u) \in \R^n \times L^2(\sp{0,T}, R^m)$ by
\begin{equation*}
\norm{\xi}  = || x_0||_2 + ||u||_2.
\end{equation*}
Given $\delta >0$ and $\xi \in \mc{X}$, $B^\delta(\xi)$ denotes the ball of radius $\delta$ centered at $\xi$.

\section{Piecewise Smooth Control Systems}
\label{sec:pws}
We now introduce the class of bimodal piecewise-smooth control systems considered in this paper, briefly review Filippov's convention for defining the dynamics of this class of systems, and introduce the smooth approximations we study throughout the paper. 
\subsection{Piecewise-smooth Control Systems}
To begin, let $g \colon \R^n \to \R$ be a regular, smooth map with Lipschitz continuous first and second partial derivatives. Consider the two disjoint domains
\begin{equation*}
{D_1 = \set{x \in \R^n \colon g(x) < 0}} 
\end{equation*}
and
 \begin{equation*}
 D_{2} = \set{x \in \R^n \colon g(x)>0},
 \end{equation*}
which are separated by the co-dimension-1 sub-manifold
 \begin{equation*}
 \Sigma = \set{x \in \R^n \colon g(x) = 0}.
 \end{equation*}
 
Let $f_1, f_2 \colon \R^n \times \R^m \to \R^n$ be smooth vector fields with Lipschitz continuous first and second partial derivatives, and then consider the piecewise-smooth control system
\begin{equation}\label{eq:pws_system}
\dot{x} = f(x,u) = \begin{cases}
f_1(x,u) & \text{ if } x \in D_1\\
f_2(x,u) & \text{ if } x \in D_2,
\end{cases}
\end{equation}
where $f \colon \R^n \times \R^m \to \R^n$. Note that $f$ is undefined along $\Sigma$, and in general may also be discontinuous along this surface. Due to the discontinuity in $f$, classical (or Caratheodory) solutions for the differential equation \eqref{eq:pws_system} may fail to exists along the surface of discontinuity. Thus, we turn to Filippov's convention to define the dynamics of the system. 
\subsection{Filippov Solutions}

We now review Filippovs convention, and refer the reader to \cite{filippov2013differential} for a more thorough introduction. Throughout the document we will make the following regularity assumption, which ensures the existence and uniqueness of Filippov solutions for the discontinuous system \eqref{eq:pws_system} \cite[Chapeter 2, Section 2, Theorem 2]{filippov2013differential}.

 \vspace{0.1cm}
\begin{assumption}\label{ass:transverse}
For each $(x,u) \in \Sigma \times \R^m$ either $L _{f_1}g(x,u)>0$ or $L_{f_2}g(x,u)<0$.
\end{assumption}
 \vspace{0.1cm}
The assumption rules out pathological cases where $L _{f_1}g(x,u)<0$ and $L_{f_2}g(x,u)>0$ along $\Sigma$, in which case admissible solutions may either enter $D_1$ or $D_2$. 

Filippov's convention compactly defines the dynamics of the discontinuous system by
\begin{equation}\label{eq:inclusion}
    \dot{x} \in F(x,u) =  \begin{cases}
    f_1(x, u) & \text{ if } x \in D_1\\
    \overline{co}\set{f_1(x,u),f(x,u)} & \text{ if } x \in \Sigma\\
    f_2(x,u) & \text{ if } x \in D_2,
    \end{cases}
\end{equation}
where $F \colon \R^m \times \R^n \to \mathcal{P}(\R^n)$ is a multivalued map \footnote{For a more formal derivation of this differential inclusion the reader is referred to \cite[Chapter 1, Section 2]{filippov2013differential}}. For each set of data $\xi = (x_0, u) \in \mc{X}$, we say that the absolutely continuous function $x^\xi \colon \sp{0,T} \to \R^n$ is the \emph{Filippov solution} corresponding to this data if $x^{\xi}(0) = x_0$ and it satisfies the differential inclusion \eqref{eq:inclusion} at almost everywhere. For each $t \in \sp{0,T}$ we then definine the map $\oldphi_t \colon \mc{X} \to \R^n$ by
\begin{equation}
  \oldphi_t(\xi) = x^{\xi}(t).  
\end{equation}

With Assumption \ref{ass:transverse} in effect, it is customary to distinguish the following two regions of $\Sigma \times U$:
\begin{enumerate}
    \item The \emph{crossing region}: 
    \begin{equation*}
        \set{(x,u) \in \Sigma \times U \colon (L_{f_1}g(x,u))\cdot (L_{f_2}g(x,u)) >0}
    \end{equation*}
    \item The \emph{sliding region}:
    \begin{equation*}
        \{(x,u) \in \Sigma \times U \colon L_{f_1}g(x,u) >0 \text{ and } \   L_{f_2}g(x,u) <0\}
    \end{equation*}
\end{enumerate}
When a Filippov solution is in the crossing region, it leaves the surface of discontinuity and enters either $D_1$ or $D_2$. However, when the solution is in the sliding region the solutions \emph{slides} along the surface of discontinuity 
and obeys the differential equation
\begin{equation}
    \dot{x} = f^s(x,u),
\end{equation}
 where Filippov's famous \emph{sliding vector field} is defined by
\begin{equation*}
f^s(x,u) = (1-\alpha(x,u))f_1(x,u) +\alpha(x,u) f_2(x,u)
\end{equation*}
where 
\begin{equation*}
\alpha(x,u) = \frac{\nabla g(x)\cdot f_1(x,u)}{\nabla g(x)\cdot (f_1(x,u) - f_2(x,u))}
\end{equation*}
selects the unique convex combination of $f_1$ and $f_2$ which keeps the trajectory on $\Sigma$.

\begin{comment}
The defining characteristic of Filippov's solution concept, however, is how trajectories evolve along the surface of discontinuity. Suppose that at time $t \in \sp{0,T}$ the trajectory reaches $\Sigma$ and either $L_{f_1}g(\oldphi_t(\xi),u(t)) >0$ and $L_{f_2}g(\oldphi_t(\xi),u(t)) \geq 0$ or $L_{f_1}g(\oldphi_t(\xi),u(t)) \leq 0$ and $L_{f_2}g(\oldphi_t(\xi),u(t))<0$. In this case the solution simply crosses the surface of discontinuity, and then continues to evolve according to \eqref{eq:interior}. On the other hand, if at time $t \in \sp{0,T}$ we have $\oldphi_t(\xi) \in \Sigma$ and $L_{f_1}g(\oldphi_t(\xi))>0$  and $L_{f_2}g(\oldphi_t(\xi))<0$ then the trajectory \emph{slides} along the surface of discontinuity and obeys
\begin{equation}
\pderiv{}{t}\oldphi_t(\xi) = f^s(\oldphi_t(\xi),u(t))
\end{equation}

is chosen such that $L_{f^s}g(x,u) =0$, keeping the trajectory on the surface. The solution stays on the trajectory until either $L_{f_1} g$ or $L_{f_2}g$ changes sign, at which point the trajectory leaves the surface of discontinuity and flows into either $D_1$ or $D_2$. We next introduce the family of smooth approximations which we use to approximate the Filippov solutions of \eqref{eq:pws_system}.

\end{comment}

\subsection{Smooth Approximations}
\label{sec:smoothed}
Our family of smooth approximations are parameterized by $\epsilon >0$ and are obtained by smoothing \eqref{eq:pws_system} along the \emph{region of regulatization}
\begin{equation*}
\Sigma^\epsilon = \set{x \in \R^n \colon -\epsilon < g(x) < \epsilon}.
\end{equation*}
The main idea behind the smoothing approach is to use the following class of functions to transition between $f_1$ and $f_2$ along $\Sigma^\epsilon$.
 \vspace{0.1cm}
\begin{definition} 
We say that $\phi \in C^{\infty}(\R, [0,1])$ is a \emph{transition function} if 1) $\phi(a) =0$ if $a \leq -1$,
 2) $\phi(a)= 1$ if $a \geq 1$, 3) $\phi$ is strictly increasing on $(-1,1)$, and 4) each of the first and second partial derivatives of $\phi$ are Lipschitz continuous.
\end{definition}
 \vspace{0.1cm}
For the rest of the paper, we assume that as single transition function $\phi$ has been selected. For each $\epsilon>0$ we then define the $\epsilon$-relaxation of \eqref{eq:pws_system} to be
\begin{equation}\label{eq:smooth}
    \dot{x} = f^\epsilon(x,u)
\end{equation}
where $f^\epsilon \colon \R^n \times \R^m \to \R^n$ is given by
\begin{equation*}\label{eq:smooth_system}
f^{\epsilon}(x,u) = \p{1 - \phi\p{\frac{g(x)}{\epsilon}}} f_1(x,u) + \phi\p{\frac{g(x)}{\epsilon}}f_2(x,u).
\end{equation*}
Note that $f^\epsilon(x,u)$ is a convex combination of $f_1(x,u)$ and $f_2(x,u)$ if $x \in \Sigma^\epsilon$, and that $f^{\epsilon}(x,u) = f_1(x,u)$ if $x \in D_1 \setminus \Sigma^\epsilon$ and $f^\epsilon(x,u) = f_2(x,u)$ if $x \in D_2 \setminus \Sigma^\epsilon$. 

It is straightforward to see that for each $\epsilon >0$ the vector field $f^\epsilon$ is smooth and Lipschitz continuous, and thus the trajectories corresponding to the smooth system will be unique. Thus for each $\epsilon>0$ and $\xi =(x_0,u) \in \mc{X}$ let $x^{(\epsilon, \xi)} \colon \sp{0,T} \to \R^n$ be the solution to \eqref{eq:smooth} with initial condition $x^{(\xi,\epsilon)}(0) = x_0$ and for each $t \in \sp{0,T}$ define $\oldphi_t^\epsilon \colon \mc{X} \to \R^n$ by 
\begin{equation*}
\oldphi_t^\epsilon(\xi) = x^{(\epsilon, \xi)}(t).
\end{equation*}

The following result from \cite{westenbroek2018new} establishes that the trajectories of the relaxed system converge to the Filippov solutions of the discontinuous system as $\epsilon \to 0$ when Assumption \ref{ass:transverse} holds. 
  \vspace{0.1cm}

\begin{lemma}(\cite{westenbroek2018new})\label{lemma:convergence1}
Let Assumption \ref{ass:transverse} hold for \eqref{eq:pws_system}. Then there exists $C > 0$ and $\epsilon_0 > 0$ such that  for each $\epsilon \leq \epsilon_0$, $t \in \sp{0,T}$ and $\xi =(x_0,u) \in \mc{X}$
\begin{equation*}
\norm{\oldphi_t(\xi) - \oldphi^\epsilon_t(\xi)} \leq C \epsilon.
\end{equation*}
\end{lemma}
 \vspace{0.2cm}

In other words, the solutions to the regularized systems converge to the solutions of the non-smooth system at a rate that is linear in $\epsilon$.

\section{Optimal Control Problems}
\label{sec:ocp}
In this section we formulate an optimal control problem over the piecewise-smooth system introduced in Section \ref{sec:pws}, and a family of approximate problems formulated over the smooth regularizations introduced in Section \ref{sec:smoothed}. We then discuss the challenges associated with applying standard gradient-based algorithms to arrive at minimizers for these problems. 

\subsection{Optimal Control Problems}
 To begin, let $\ell \colon \R^n \to \R$ be a cost function which is assumed to have Lipschitz continuous first and second partial derivatives, and define the cost functional $L \colon \mc{X} \to \R$ by
\begin{equation*}
L (\xi) = \ell(\oldphi_T(\xi)).
\end{equation*}
We then define the following optimal control problem:

\begin{equation*}
\mathbf{(P)} \quad \inf_{\xi \in \mc{X}} L(\xi)
\end{equation*}
Note that this problem formulation is quite general, since there exists well-known transformations to include additional terms, such as running costs, into the terminal cost functional we consider here \cite[Chapter 5]{polak2012optimization}.

Next, for each $\epsilon >0$ we define the regularized cost functional $L^\epsilon \colon \mc{X} \to \R$ by
\begin{equation}
L^\epsilon(\xi) = \ell(\oldphi_T^\epsilon(\xi))
\end{equation}
and subsequently define the $\epsilon-$relaxation of $P$  by: 
\begin{equation*}
\mathbf{(P^\epsilon)} \quad \inf_{\xi \in \mc{X}} L^\epsilon(\xi)
\end{equation*}
We next discuss derivative based approaches for solving these optimization problems. 
\subsection{Directional Derivatives}
The majority of practical tools for finding local minimizers of optimal control problems rely on calculating the derivatives of the cost functional with respect to changes in the initial condition an inputs applied to the system. Once these derivatives are obtained, standard optimization procedures can be applied to iteratively improve performance \cite{polak2012optimization}.
\subsubsection{Directional Derivatives of Regularized Problems}
 Since for each $\epsilon >0$ the vector field $f^\epsilon$ is smooth, we can compute directional derivatives of $L^\epsilon$ using standard approaches. In particular, given $\xi = (x_0, u) \in \mc{X}$, we define $DL^\epsilon(\xi ; \cdot) \colon \mc{X}' \to \R$ by
 \begin{equation}
     DL^\epsilon(\xi; \delta \xi ) = \lim_{\lambda \downarrow 0} \frac{L^\epsilon(\xi + \lambda \cdot \delta \xi) - L^\epsilon(\xi)}{\lambda}
 \end{equation}
 for each direction $\delta \xi = (\delta x_0, \delta u) \in \mc{X}'$, which by \cite[Theorem 5.6.8]{polak2012optimization} is given by
 \begin{equation*}
     DL^\epsilon(\xi ; \delta \xi) = p(0)^T \cdot \delta x_0 + \int_{0}^T p(\tau)^T B(\tau) \delta u(\tau) d\tau
 \end{equation*}
 where $p \colon \sp{0,T} \to \R^n$ is the solution to the \emph{co-state} or \emph{adjoint} equations
\begin{equation*}
- \dot{p}(t) = A(t)^T p(t) \text{ a.e. } t \in \sp{0,T}, \  p(T)^T = \deriv{}{x}l(\oldphi_T(\xi))     
\end{equation*}
 and for each $t \in \sp{0,T}$ we have $A(t) = \pderiv{}{x}f^{\epsilon}(\oldphi_t(\xi), u(t))$ and $B(t) =\pderiv{}{u}f^\epsilon(x,u)$. In particular $DL^\epsilon(\xi; \delta \xi)$ provides a first-order approximation to how $L^\epsilon$ changes by perturbing the data supplied to the system in the direction $\delta \xi$.

 However, inspecting $f^\epsilon$, we see that the partial derivative $\pderiv{}{x}f^\epsilon(x,u)$ will be of the order $\frac{1}{\epsilon}$ when $x \in \Sigma^\epsilon$. Thus, we should be concerned that the adjoint equations associated to the regularized system will "blow up" as we take $\epsilon \to 0$ along trajectories which pass through the region of regularization. In other words, in the numerical setting we should be concerned that discretizations of $P^\epsilon$ will become ill-conditioned for small values of the regularization parameter. However, our subsequent analysis will demonstrate that this is not the case, and that the gradients of $P^\epsilon$ remain bounded even as we take $\epsilon \to 0$.

 \subsubsection{Directional Derivatives of Non-smooth Problem}
It is well-known that Filippov solutions for the discontinuous system \eqref{eq:pws_system} may have a non-smooth dependence on the initial conditions and inputs supplied to the system \cite[Chapter 2, Section 11]{filippov2013differential}. In particular, the map $\oldphi_t(\cdot)$ may fail to be differentiable at the point $\xi \in \mc{X}$ if the trajectory associated to this data arrives at the surface of discontinuity at time $t$. In the context of optimal control, this means that the adjoint equations associated to the discontinuous system may under go discrete "jumps" at time instances when the nominal trajectory reaches the surface of discontinuity\footnote{The "jump conditions" associated with the adjoint equations of the discontinuous system have a rich geometric interpretation which we do not have space to discuss in this article. However, the interested reader is referred to \cite[Chapter 2, Section 11]{filippov2013differential}) for a comprehensive discussion regarding these "jump conditions" in the context of sensitivity analysis.}. Due to these discontinuities, the adjoint equations, and consequently the gradients of $L$, are difficult to approximate numerically using standard integration techniques such as Euler integration \cite{stewart2010optimal}. As discussed in Section \ref{sec:intro}, this has typically meant that derivative-based methods for solving optimal control problems over discontinuous systems have required restrictive practices such as fixing the number of times the trajectory crosses the surface of discontinuity \emph{a priori}. This is our primary motivation for employing the smooth approximations considered in the paper.

Before proving our convergence results in Section \ref{sec:convergence}, we first formalize conditions which will ensure the differentiability of $L$ at a nominal choice of $\xi =(x_0, u) \in \mc{X}$. Each of these additional assumptions are standard when discussing the differentiability of Filippov solutions \cite[Chapter 2, Section ll]{filippov2013differential}. For compactness of notation, given $\xi =(x_0,u) \in \mc{X}$, we define $\mc{T}^{\xi}$ to be the set of time instances at which the Filippov solution corresponding to this data arrives at the surface of discontinuity:
\begin{equation*}
\mc{T}^{\xi}  = \{\hat{t} \in (0,T] \colon \oldphi_{\hat{t}}(\xi) \in \Sigma \text{ and } \oldphi_{\hat{t}^{-}}(\xi) \not \in \Sigma\} 
\end{equation*}

First, we assume that the nominal trajectory does not begin on the surface of discontinuity or arrive at it at exactly time $T$. The first condition can lead to non-differentiability with respect to initial conditions, while the later condition can cause $\oldphi_T(\cdot)$ to be non-differentiable, and consequently also the terminal cost functional $L$ 
\cite[Chapter 2, Section ll]{filippov2013differential}.
\vspace{0.2cm}
\begin{assumption}\label{ass:diff1}
The data $\xi  = (x_0, u) \in \mc{X}$ is such that $x_0 \not \in \Sigma$ and $T \not \in \mc{T}^{\xi}$.
\end{assumption}
\vspace{0.2cm}
 Our next assumption ensures that the nominal trajectory arrives at the surface of discontinuity transversely, without "skimming" the surface, a situation which is well known to cause non-differentiability. 
\vspace{0.2cm}
 \begin{assumption}\label{ass:diff3}
 The data $\xi = (x_0, u) \in \mc{X}$ is such that if $\hat{t} \in \mc{T}^\xi$ then $\exists \gamma >0$ and interval $I =(\hat{t} - \gamma, \hat{t} + \gamma) \cap \sp{0,T}$ such that

 \begin{enumerate}
 \item $\oldphi_{\hat{t}^-}(\xi) \in D_1 \implies$ $L_{f_1}g(\oldphi_t(\xi), u(t))>0$,  $ \forall t \in I$
 \item $\oldphi_{\hat{t}^-}(\xi) \in D_2 \implies$  $L_{f_2}g(\oldphi_t(\xi), u(t)) <0$, $ \forall t \in I$ \end{enumerate}
 \end{assumption}
 \vspace{0.2cm}
 The final assumption we make rules our pathological cases such as "Zeno" behavior, where the trajectories of the system cross the surface of discontinuity an infinite number of times. 
 
\begin{assumption}\label{ass:diff5}
The data $\xi = (x_0, u) \in \mc{X}$ is such that $\mc{T}^{\xi}$ is a finite set. 
\end{assumption}

When the data $\xi \in \mc{X}$ satisfies Assumptions \ref{ass:diff1}-\ref{ass:diff5}, the map $\oldphi_T(\cdot)$ will be differentiable at $\xi$, and thus the directional derivatives of $L$ will be well-defined at this point. When this is the case we define $DL(\xi ; \cdot)\colon \mc{X}' \to \R$ by
\begin{equation}
DL(\xi ; \delta \xi ) = \lim_{\lambda \downarrow 0} \frac{L(\xi + \lambda \cdot \delta \xi) - L(\xi)}{\lambda},
\end{equation}
for each each direction $\delta \xi \in \mc{X}$.

\section{Convergence Results}
\label{sec:convergence}
We are now ready to present the theoretical contribution of this work. In section \ref{sub:convergence} we demonstrate conditions under which the directional derivatives of the regularized problem converge to directional derivatives on the discontinuous problem, and in Section \ref{sub:minimizers} we use this result to discuss how the regularized problems can be used to consistently approximate the non-smooth problem. 

\subsection{Convergence of Directional Derivatives}\label{sub:convergence}
We first demonstrate that when Assumptions \ref{ass:diff1}-\ref{ass:diff5} are satisfied, and directional derivatives on the non-smooth cost functional are well-defined, that that the regularized problems can be used to approximate these derivatives with arbitrary precision. We first develop a few intermediate lemmas which are used in the proof of the main result. Proofs for the lemmas can be found in the appendix. The first lemma demonstrates that variations on the relaxed problems remain bounded, even as $\epsilon \to 0$. 
 
\begin{lemma}\label{lemma:bounded}
Let Assumption \ref{ass:transverse} hold. Further, assume that the data $\xi = (x_0, u) \in \mc{X}$ satisfies Assumptions \ref{ass:diff1}-\ref{ass:diff5} hold. Then there exists $C>0$ and $\epsilon_0>0$ such that for each $\epsilon \leq \epsilon_0$ and $\delta \xi = (\delta x_0, \delta u)\in \mc{X}$
\begin{equation}
\norm{DL^{\epsilon}(\xi;  \delta \xi) - DL^\epsilon(\xi ; \delta \xi)} \leq C \norm{\delta \xi}. 
\end{equation}
\end{lemma}

Next, we establish our final two preliminary results needed for the statement of our firs convergence Theorem. The next result demonstrates conditions under which directional derivatives on the regularized problems converge to a well defined limit as $\epsilon \to 0$, and will provide the rate of convergence in Theorem \ref{thm:convergence1}. We remind the reader the the proof of the following two results are contained in the appendix.

\begin{lemma}\label{lemma:Lipschitz}
Let Assumption \ref{ass:transverse} hold. Let the data $\xi \in \mc{X}$ satisfy assumptions \ref{ass:diff1}-\ref{ass:diff3}. Then there exists $C>0$ and $\epsilon_0 >0$ such that for each $\epsilon_1, \epsilon_2 \leq \epsilon_0$ and $\delta \xi \in \mc{X}$ we have
\begin{equation}\tag{73}
\norm{DL^{\epsilon_1}(\xi ;  \delta \xi) - DL^{\epsilon_2}(\xi ;  \delta \xi)} \leq C|\epsilon_1 -\epsilon_2|.
\end{equation}
\end{lemma}

The final preliminary result demonstrates that if Assumptions \ref{ass:diff1}-\ref{ass:diff5} hold at the point $\xi \in \mc{X}$, then they also hold in a neighborhood of $\xi$.

\begin{lemma}\label{lemma:neighborhood}
Let Assumption \ref{ass:transverse} hold. Let Assumptions \ref{ass:diff1}-\ref{ass:diff5} hold for the data $\xi \in \mc{X}$. Then there exists $\delta >0$ such that assumptions \ref{ass:diff1}-\ref{ass:diff3} also hold for each $\xi' \in B^\delta(\xi)$. 
\end{lemma}
With these preliminary results established, we are ready to state our first convergence result.
\begin{theorem}\label{thm:convergence1}
Let Assumption \ref{ass:transverse} hold. Let $\xi =(x_0, u) \in \mc{X}$ satisfy Assumptions \ref{ass:diff1}-\ref{ass:diff5}. Then there exists $C >0$ and $\epsilon_0 > 0$ such that for each $\delta \xi = (\delta x_0, \delta u) \in \mc{X}'$ and $\epsilon \leq \epsilon_0$
\begin{equation*}
||DL(\xi; \delta \xi) - DL^\epsilon(\xi; \delta \xi)|| \leq C \epsilon.
\end{equation*}  
\end{theorem}
\begin{proof}
Let $\set{\epsilon_i}_{i \in \N}$ be a sequence such that $\epsilon_i \to 0$ as $i \to \infty$. By Lemma \ref{lemma:convergence1} and the continuity of $\ell$, we see that the sequence of cost functionals $\set{L^{\epsilon_i}}_{i \in \N}$ converges point-wise to $L$. Using a standard uniform convergence argument (see e.g. \cite[Theorem 8.6.3]{dieudonne2013foundations}), we know that if we can show that there exists $\delta >0$ such that for each $\xi' \in B^\delta(\xi)$ the sequence of linear maps $\set{ DL^{\epsilon_i}(\xi' ; \cdot)}_{i \in \N}$ is bounded and converges uniformly, then its limit point must be $DL(\xi' ; \cdot)$. In other words, this would demonstrate that directional derivatives of $\set{L^{\epsilon_i}}_{i \in \N}$ converge to the directional derivatives of $L$ in a $\delta$-neighborhood of $\xi$. Let $\delta \xi \in \mc{X'}$ be arbitrary. By Lemmas 2, 3 and 4, we know that there exists constants $C_1,C_2>0$, $\epsilon_0 >0$, and $\delta >0$ such that for each $\epsilon \leq \epsilon_0$ and $\xi' \in B^{\delta}(\xi)$
\begin{equation}
    DL^\epsilon(\xi' ; \delta \xi) \leq C_1 \norm{\delta \xi}
\end{equation}
and $\bar{\epsilon}_1, \bar{\epsilon}_2 \leq \epsilon_0$ we have
\begin{equation}\label{eq:lipschitz1}
    \norm{DL^{\bar{\epsilon}_1}(\xi'; \delta \xi) - DL^{\bar{\epsilon}_2} (\xi'; \delta \xi)} \leq C_2 |\bar{\epsilon}_1 - \bar{\epsilon}_2|,
\end{equation}
which demonstrates that the sequence of linear maps $\set{DL^\epsilon(\xi' ; \cdot)}_{i \in N}$ converges to a bounded limit as $ i \to \infty$, as desired. The linear rate of convergence follows from \eqref{eq:lipschitz1}. 
\end{proof}

At this point it is prudent to compare Theorem \ref{thm:convergence1} to the convergence results presented in \cite{stewart2010optimal}. First we note that our regularity Assumption \ref{ass:transverse} is weaker than the one-sided Lipschitz assumption made in \cite[Section 3.1]{stewart2010optimal}. Next, we note that we provide a concrete rate of convergence in Theorem \ref{thm:convergence1}. Note that Theorem \ref{thm:convergence1} also demonstrates that the \emph{equations of sensitivity} \cite[Chapter 2, Section ll]{filippov2013differential} for the regularized system converge, at a rate that is linear in $\epsilon$, to the equations of sensitivity of the non-smooth system. Many applications, such as assessing the stability of periodic orbits \cite{westervelt2003hybrid}, rely on accurately assessing the equations of sensitivity with known margins of error. We anticipate that the bound in Theorem \ref{thm:convergence1} will also find use in such settings. 
\subsection{Consistent Approximations}\label{sub:minimizers}
In practice, solutions to difficult optimization problems are often approximated by solving a \emph{sequence} of approximate problems. The most common example are penalty methods, wherein a sequence of increasingly penalized unconstrained optimization problems are used to find minimzers for a difficult constrained optimization problem.  More generally, analyzing sequential algorithms of this sort falls into the framework of \emph{consistent approximations} as stated by Polak \cite[Chapter 3.3]{polak2012optimization}. In this section, we discuss how this theory can be applied to find minimizers of $P$ using so-called "Master Algorithms" and a sequence of regularized problems $\set{P^{\epsilon_i}}_{i \in \N}$.

Even though, as we saw in the proof of Theorem \ref{thm:convergence1}, $L^\epsilon$ converges to $L$ as $\epsilon \to 0$, this is in general not sufficient to show that a sequence of minimizers to the sequence $\set{P^\epsilon_i}_{i \in N}$ accumulate to a minimizer of $P$ \cite[Chapter 3.3]{polak2012optimization}. In particular, the consistency results we pursue will rely on comparing \emph{optimality functions} for the non-smooth and regularized problems, which provide additional derivative information about the problems. 

\begin{definition}
Let $G$ be an optimization problem on the normed space $S$. We say that $\theta \colon (-\infty, 0]$ is an \emph{optimality function} for $G$ iff it is continuous and $\theta(\bar{x}) = 0$ for each minimizer $\bar{x} \in S$.
\end{definition}

Given a difficult optimization problem and optimality function pair $(G, \theta)$, the theory of consistent approximations relies on finding a sequence of approximate optimization problems and optimality functions $\set{(G_n,\theta_n )}_{n \in \N}$ and being able to show that the epigraphs of the pair $(G_n,\theta_n)$ converge to the epigraphs of $(G, \theta)$ as $N \to \infty$. Then one can deploy the aforementioned Master Algorithms \cite[Chapter 3.3.3 ]{polak2012optimization} to solve of $G$ by iteratively optimizing a subsequence of $\set{G_n}_{n \in \N}$. Due to space constraints, we leave a more detailed introduction to \cite{westenbroek2} or \cite{polak2012optimization}, and instead discuss here how to obtain such consistency results for the problems we consider. 

Due to the non-differentiability of $L$, we modify the approach typically used to demonstrate such consistency results for approximations to smooth optimal control problems \cite[Chapter 4]{polak2012optimization}. In particular, we will restrict our attention to regions where Assumptions \ref{ass:diff1}-\ref{ass:diff5} are satisfied and the regularized problems provide accurate derivative information about the non-smooth problem in the limit.
We propose the following optimality functions for the smooth and non-smooth problems, which are restricted to subsets of $\mc{X}$ on which Assumptions \ref{ass:diff1}-\ref{ass:diff5} are satisfied. 
\begin{proposition}\label{prop:optimality}
Let $S \subset \mc{X}$ be a connected, convex set such that Assumptions \ref{ass:diff1}-\ref{ass:diff5} are satisfied for each $\xi \in \mc{S}$. Consider the restricted optimization problems:
\begin{equation}
    (\mathbf{P}_S) \quad \inf_{\xi \in S} L(\xi),
\end{equation}
\begin{equation}
    (\mathbf{P}_{S}^\epsilon) \quad \inf_{\xi \in S} L^\epsilon(\xi)
\end{equation}
Then the following two functions are optimality functions for $P_{S}$ and $P_{S}^\epsilon$, respectively: 
\begin{equation}
    \theta_S(\xi) = \inf\set{DL(\xi; \delta \xi) \colon \delta \xi \in \mc{X}' \text{ and }  \xi + \delta \xi \in \mc{X}}
\end{equation}
\begin{equation}
   \theta_{S}^\epsilon(\xi) = \inf\set{DL^\epsilon(\xi; \delta \xi) \colon \delta \xi \in \mc{X}' \text{ and } \xi + \delta \xi \in \mc{X}}
\end{equation}
\end{proposition}
\begin{proof}
We omit a detailed proof of the fact that the proposed optimality functions appropriately capture the minimizers of the respective optimization problems, as the argument closely follows the arguments made \cite[Theorem 5.6.9]{polak2012optimization} and \cite[Proposition 4.5]{polak1984study}. The continuity of the relaxed optimality function follows by noting that time map $ \xi \to DL^\epsilon(\xi, \delta \xi)$ is continuous for each $\epsilon >0$ \cite[ Theorem 5.6.9]{polak2012optimization}, and recalling Danskin's Theorem. The continuity of the proposed optimality function for the non-smooth problem follows from the uniform convergence displayed in Theorem \ref{thm:convergence1}.
\end{proof}

We can now state the following consistency result:
\begin{theorem}
Let $\set{\epsilon_i}_{i \in \N}$ be a sequence such that $\epsilon_i \to 0$ as $i \to \infty$, and let $S$ be as in Proposition \ref{prop:optimality}. Then the pair $(L_S^{\epsilon_i},\theta_S^{\epsilon_i})$ epi-converges to $(P_S, \theta_S)$ as $i \to \infty$.
\end{theorem}
\begin{proof}
The pointwise convergence of $\set{L_S^{\epsilon_i}}_{i \in \N}$ to $L_S$ noted in the proof of Theorem \ref{thm:convergence1} also demonstrates epi-convergence of the sequence. The epi-convergence of $\set{\theta_S^{\epsilon_i}}_{i \in \N}$ to $\theta_S$ follows directly from the statement of Theorem \ref{thm:convergence1}.
\end{proof}

This result demonstrates that the appropriate consistency results from \cite[Chapters 3-4]{polak2012optimization} hold, so long as we restrict our attention to  subsets where Assumptions \ref{ass:diff1}-\ref{ass:diff5} hold. In practice however, we do not want to restrict ourselves ahead of time to a single compact subset $S$ on which Assumptions \ref{ass:diff1}-\ref{ass:diff5} are satisfied. In particular, we actually want to optimize over all of $\mc{X}$ and iteratively optimize over a sequence $\set{L^{\epsilon_i}}_{i \in \N}$, in which case we should expect that successive iterates will sometimes produce initial conditions which cross the surface of discontinuity between iterations, or generate trajectories whose final state crosses the surface between iterations. Nonetheless, so long as an algorithm which successively optimizes over $\set{L^\epsilon_i}_{i \in \N}$ accumulates to a point which satisfies Assumptions \ref{ass:diff1}-\ref{ass:diff5}, we can invoke Theorem 2 in a neighborhood of this point to study the limiting behavior of the algorithm. Though we plan to provide a more exhaustive exploration of this topic and introduce implementable numerical algorithms for solving $\mathbf{P}$ in a forthcoming article, we believe the results we have indicated here provide a valuable basis for this important future step.  

\begin{comment}
\section{Application to Bipedal Hopping via Reduced-order Modelling}
\label{section:example}
\input{example}
\end{comment}

\section{Trajectory Optimization via Smoothing on a Reduced-order Model}
\label{section:example1}
To demonstrate the utility of the smoothing technique we have investigated, we first use the approach to generate hopping trajectories for an actuated Spring-mass hopper, which was first proposed in \cite{xiong2018bipedal} for realizing the hopping on the bipedal robot Cassie. Its nonlinear leg spring model has also shown great value for other robotic applications such as underactuated bipedal walking \cite{xiong2018coupling} and fully-actuated humanoid walking \cite{xiong2019exo}. We note, however, that the trajectory generation procedures proposed in these prior works required fixing the sequence of contacts the robot makes with the ground \emph{a priori}. Here, we demonstrate that the regularization techniques we consider can be used to effectively generate optimal trajectories on the Spring-mass model without pre-specifying the contact sequence ahead of time.

\begin{comment}To achieve the result, we combine the approach discussed here with recently proposed techniques which employ embedded reduced-order models for motion planning \cite{xiong2018bipedal}, \cite{xiong2019exo}. At a high level, we use the smoothing techniques presented here to generate desired behaviors on an actuated Spring-mass model of the robot, and then use established feedback techniques to track these trajectories on the full-order robot model. We begin by introducing the actuated Spring-mass model we employ. 
.  In the future we hope to extend the proof of concept we present here to generate walking motions.
\end{comment}

\begin{comment}
We primarily apply the smoothing technique on the contact implicit optimization of an actuated Spring-mass hopper \cite{xiong2018bipedal} and then utilize dynamics embedding \cite{xiong2019exo} to realize the hopping on a full 3D robot model of an exoskeleton (Fig. \ref{}). 
\end{comment}

\subsection{Actuated Spring-mass Model of Hopping}\label{sub:slip}
The actuated Spring-mass hopper is depicted in Fig. \ref{results1} (a). The model has four states: the height of the mass, $z$, and its time derivative $\dot{z}$, as well as the natural length of the leg spring, $L$, and its time derivative, $\dot{L}$. The model is actuated by controlling $L$, i.e., 
\begin{equation}
\ddot{L} = u,
\end{equation}
where $u$ is assumed to be the actuation on the leg. The \emph{Flight} phase is ballistic, thus the dynamics of the mass is,
\begin{align}
\ddot{z} = -g
\end{align}
where $g$ is the gravitational constant. When the hopper is on the ground, the dynamics of the mass is,
\begin{align}
\ddot{z} &= \frac{F(z, \dot{z},L, \dot{L})}{m} -g
\end{align}
where $m$ is the value of the mass and $F(z, \dot{z},L, \dot{L}) =  K(L)(L-z) + D(L) (\dot{L} - z)$ is the leg force from the spring deflection. Here, $K(L)$ and $D(L)$ are the stiffness and damping coefficient, respectively. For further details of the model can be found in \cite{xiong2018bipedal}. The hopper lifts off the ground with $F\rightarrow0$. In practice, the $K(L)$ is much larger than $D(L)$. Thus we simplify the assumption that the robot lifts off when $L = z$. \begin{comment}
The validity of this assumption, which was not originally made in \cite{xiong2018bipedal}, is validated by the tight tracking we see between the Spring-mass and full-order models in Section \ref{sec:simulation2}. 
\end{comment}
Collecting the states of the robot as $x = (z, \dot{z}, L , \dot{L})$, we represent the dynamics using the piecewise-smooth vector field $f \colon \R^4 \times \R \to \R^4$ where,
\begin{equation}
\dot{x} = f(x,u) = \begin{cases}
f_\text{F}(x,u) & \text{ if } g(z) >0 \\
f_\text{G}(x,u) & \text { if } g(z) < 0,
\end{cases}
\end{equation}
where, 
\begin{align}
f_\text{F}(x,u) &=(\dot{z}, -g, \dot{L}, u)^T \\
f_\text{G}(x,u) &=(\dot{z},F(x, \dot{z}, L, \dot{L})/m -g, \dot{L}, u)^T \\
g(x) &= z - L.
\end{align}

\subsection{Trajectory Optimization Via Smoothing}
Now we apply the proposed smoothing method for optimizing a hopping motion without specifying the contact modes. The hopping task we choose is to reach an apex height $z_{\text{apex}}=1$m at $t_{\text{apex}}=1$s from a static standing configuration, and then settle back to its original height at $t_{\text{f}}=1.8$s. The initial condition is set as $x_0 = (z(0), \dot{z}(0), L(0), \dot{L}(0) )^T = (.65, 0 , .75, 0 )^T$. The input is bounded, i.e., $u\in U = [-10, 10]$. The cost function is,
\begin{align}
    J(x)  = &(z(t_{\text{apex}}) - z_{\text{apex}})^2 + \dot{z}(t_{\text{apex}})^2 + \nonumber\\ & (z(t_{\text{f}}) - z(0))^2 + \dot{z}(t_{\text{f}})^2 + \int u^2 dt. \nonumber
\end{align}

To approximate a solution to the problem numerically, we use the regularization parameter $\epsilon = .01$, and Euler integration with 200 uniformly spaces gridpoints. The resulting finite dimensional optimization problem is solved using the Matlab fmincon function. The optimized Spring-mass trajectory is shown in Fig. \ref{results1} (b). Due to the actuation limitations of the model, the optimized hopping actually requires two hops to reach to the apex height and one additional hop to settle.

\begin{figure}[t]
      \centering
      \includegraphics[width = 2.7in]{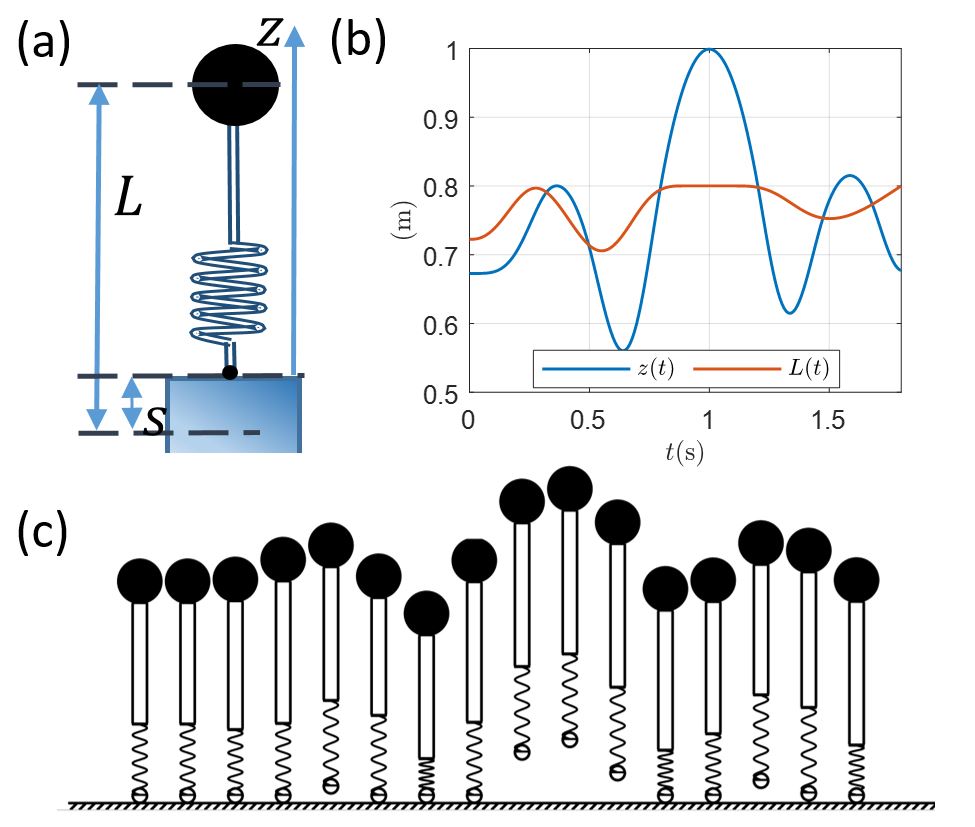}
      \caption{(a) The actuated Spring-mass model. (b) The optimized hopping trajectories. (c) Visualization of the hopping.}
      \label{results1}
\end{figure}

\section{Hopping Embedding on the Exoskeleton}
\label{section:example2}
 Despite the simplicity of the Spring-mass hopper, the optimized dynamics can be embedded onto complex robotic systems. This further amplifies the application of proposed smoothing technique. Here we briefly describe the process of embedding the hopping dynamics of the Spring-mass onto the exoskeleton (Fig. \ref{results2} (a)). More examples of the dynamics embedding from simple models to full robot models can be found in \cite{xiong2018bipedal} \cite{xiong2018coupling} \cite{xiong2019exo} and the references therein.

\subsection{Robot Model and Hybrid Dynamical Model} 
The exoskeleton (Exo) has two legs with 12 motor joints in total. Its equation of motion can be described by the floating-base Euler-Lagrange equation, 
\begin{eqnarray}
&& M\ddot{q} + h = Bu + J_{v}^T F_{v},  \label{eom} \\
&& J_{v} \ddot{q} + \dot{J}_{v} \dot{q} = 0, \label{hol}
\end{eqnarray}
where $q\in SE(3) \times \mathbb{R}^{n = 12}$.
The exact meaning of each term can be found in \cite{xiong2018bipedal}.
\begin{comment}
$M$ is the mass matrix, $h$ is the Coriolis, centrifugal and gravitational term, $B$ and $u\in \mathbb{R}^{12}$ are the actuation matrix and the motor torque vector, and $F_{v}\in \mathbb{R}^{n_{v}}$ and $J_{v}$ are the holonomic force vector from ground contact and the corresponding Jacobian respectively.
\end{comment}
For the hopping behavior, we assume the two feet always have same contact, i.e. either both feet are or are not in contact with the ground. Thus the hopping is an alternation between two domains, i.e. \textit{Ground} and \textit{Flight} domain. We use subscript $v$ to denote different domains. The guards are defined from the transition conditions, i.e. ground normal forces$\rightarrow$0 and foot position$\rightarrow$0. We model the impact as plastic impact from \textit{Flight}$\rightarrow$\textit{Ground}, where the joint velocities have jumps, i.e.,
$ \dot{q}_{\text{Ground}}^{+} = \Delta(q)\dot{q}_{\text{Flight}}^{-}.$

\subsection{Output Selection and CLF-QP based Force Control.} 

For \textit{Ground}, the robot is fully actuated and 6 outputs are required to be defined. We first select the center of mass (COM) position and the yaw and roll angles of the pelvis/torso of the robot as relative degree 2 outputs, i.e.,
\begin{equation}
 \mathcal{Y}_2^{\textrm{Ground}}(q,t) = \begin{bmatrix}
  p_{\textrm{COM}^x}(q) \\
   p_{\textrm{COM}^y}(q)  \\
   p_{\textrm{COM}^z}(q)  \\
   \mathbf{\phi}_{\textrm{pelvis}}(q) 
     \end{bmatrix} -
\begin{bmatrix}
0 \\
  0  \\
    z^{\text{opt}}(t)  \\
\mathbf{0}  \end{bmatrix},
\end{equation}
where $ z^{\text{opt}}(t)$ is the mass trajectory of the Spring-mass model for the embedding. It is also necessary to have zero centrodial angular momentum before \textit{Flight} phase \cite{xiong2018bipedal}. We select the last output as, 
\begin{equation}
\mathcal{Y}_1^{\textrm{Ground}}(q,t) = H_\text{pitch}(q, \dot{q}) - 0,
\end{equation}
where $H_\text{pitch}(q, \dot{q})$ is the pitch centrodial momentum.

In \textit{Flight}, the robot is underactuated. Here, 12 outputs are required to be defined since $n_u=12$. We selected the outputs as the positions of the feet to the COM and foot orientations. The desired vertical positions between the COM and the feet are the real leg length $L - z$ of the Spring-mass hopper. 
\begin{equation}
 \mathcal{Y}_2^{\textrm{Flight}}(q,t) = \begin{bmatrix}
  \mathbf{p}_{\text{Feet} \rightarrow \text{COM}}^{x}(q) \\
  \mathbf{p}_{\text{Feet} \rightarrow \text{COM}}^{y}(q) \\
  \mathbf{p}_{\text{Feet} \rightarrow \text{COM}}^{z}(q) \\
   \mathbf{\phi}_{\textrm{Feet}}(q)     \end{bmatrix} -
\begin{bmatrix}
\mathbf{0} \\
  \mathbf{c}  \\
   L^{\text{opt}}(t) - z^{\text{opt}}(t)  \\
\mathbf{0} \end{bmatrix},
\end{equation}
where $\mathbf{c}$ is a constant vector and its value depends on the initial state of the flight phase.  

We apply the control Lyapunov function based Quadratic programs (CLF-QP) \cite{ames2013towards} \cite{xiong2018coupling} for feedback-zeroing the outputs. The exponential convergence on the output dynamics is enforced by the inequality condition on the constructed Lyapunov function, i.e., 
\begin{equation}
\dot{V}(u, q, \dot{q}) \leq - \gamma V(q, \dot{q}),
\end{equation}
with $\gamma >0$. This inequality is affine with respect to $u$,
\begin{equation}
A_v^{\textrm{CLF}}(q,\dot{q}) u \leq b_v^{\textrm{CLF}}(q,\dot{q}).
\end{equation}
For external contacts as holonomic constraints, the holonomic forces $F_v$ are affine with respect to $u$, 
\begin{equation}
 F_v = A_v u + b_v, \label{utoF}
\end{equation}
where $A_v, b_v$ can be found in \cite{xiong2019exo}. Thus contact constraints such as friction cones, non-negative normal forces can be formulated as an inequality constraints on $u$, 
\begin{equation}
C_v A_v u \leq -C_v b_v,
\end{equation}
where $C_v$ is a constant matrix. Since it is desirable to embed the normal force of the Spring-mass hopping on the full robot, an equality constraint can be expressed for the force control, 
\begin{equation}
\label{ForceEquality}
S_v A_v u = F^{\text{opt}}  -S_v b_v,
\end{equation}
where $S_v$ is a selection matrix to extract the vertical normal forces from the holonomic force vector and $F^{\text{opt}}$ is the leg force from the Spring-mass optimization (Fig. \ref{results2} (b)). It is necessary to relax the force control since the vertical COM position is one of the outputs for control \cite{xiong2019exo}. Thus the equality in Eq. \eqref{ForceEquality} becomes,
\begin{equation}
\label{ForceRelaxation}
  \underset{c_{lb}}{\underbrace{(1-c)  F^{\text{opt}} - S_v b_v}} \leq  S_v A_v u   \leq \underset{c_{ub}}{\underbrace{ (1+c)  F^{\text{opt}} - S_v b_v}},
\end{equation}
where $c\in (0,1)$ is a coefficient of the relaxation. 

The main control law for the full-order model in each domain is thus formulated as a quadratic program as follows:
\begin{align}
\label{QP}
\quad u^{*} =  \underset{u \in \mathbb{R}^{12}, \delta\in \mathbb{R}} {\text{argmin}}& u^T  H u + 2 F u + p\delta^2, \\
\text{s.t.} \quad  &  A_v^{\textrm{CLF}}(q,\dot{q}) u \leq b_v^{\textrm{CLF}}(q,\dot{q})  + \delta, \ \ \quad  \text{(CLF)}\nonumber \\
 \quad &  C_v A_v u \leq -C_v b_v,   \quad  \quad \quad \quad \quad \text{(Contact)}    \nonumber  \\
 \quad  & u_{lb} \leq u \leq u_{ub},  \nonumber  \quad  \quad  \quad   \quad   \text{(Torque Limit)}\\
 \quad  &  c_{lb} \leq   S_v A_v u \leq c_{ub}, \ \ \ \quad \text{(Force Control)}\nonumber
\end{align}
where $\delta$ is a relaxation term for increasing the instantaneous feasibility of the QP, and $p$ is a positive penalty constant. 
\begin{comment}
Since there is no holonomic constraint in flight phase, the GRF and force control constraints only exist in the ground phase. 
\end{comment}

% %%%%%%%%%%%%%%%%%%%%%%%%%%%%%%%%%%%%%%%%%%%%%%%%%%%%%%%%%%%%%
\subsection{Simulation Result}\label{sec:simulation2}
We primarily implemented the embedding in simulation to demonstrate the value of the smoothing technique. Details of the implementation can be found in \cite{xiong2019exo}, where same embedding was implemented for realizing walking on the Exo. Fig. \ref{results2} (c)(d) show the resulting hopping motion. The COM trajectory matches well in general. The exerted hopping corresponds to the Spring-mass hopper. A video of the simulation can also be found in \cite{Supplementary}.
\begin{figure*}[!t]
      \centering
      \includegraphics[width = 7in]{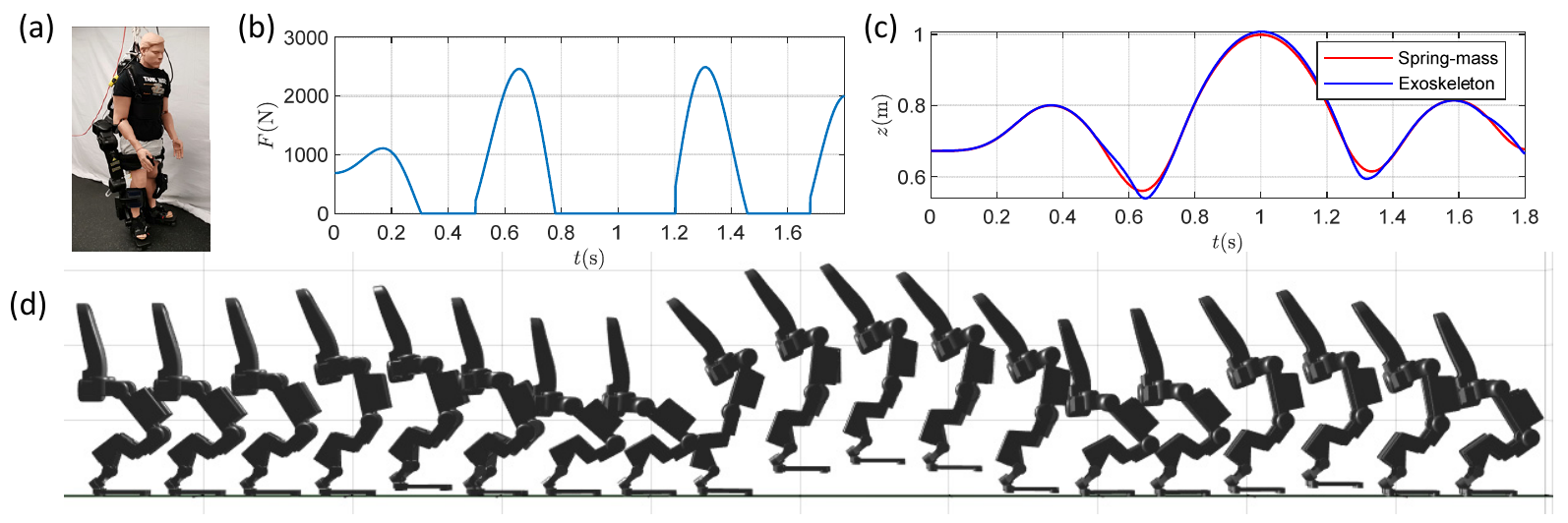}
      \caption{(a) The Exoskeleton. (b) The optimized ground force of the hopping. (c) Comparison of the COM trajectories between the Spring-mass and the Exo. (d) The snapshots of the realized hopping.}
      \label{results2}
\end{figure*}

\section{Conclusion}
This paper investigated a method for smoothing out a discontinuous differential equations. We studied optimal control problems over the discontinuous system and its regularizations, derived a number of useful properties for these problems, examined conditions under which the regularizations can be used to obtain accurate derivative information about the non-smooth problem, and  outlined how to invoked the theory of consistent approximations to study how minimizers of the regularization accumulate to minimizers of the non-smooth problem. Finally, we demonstrated the efficacy of the smoothing approach by using it in conjunction with recent reduced-order modelling techniques to generate non-trivial motion plans on a 18-DOF exoskeleton robot in simulation. 

\section{Acknowledgement}
We would like to thank Thomas Gurriet in Amber lab for providing the Simulink simulation of the Exoskeleton. 
%% References
\bibliographystyle{IEEEtran}

\appendix
\section{Technical Lemmas}
This section contains a number of technical lemmas and proofs of lemmas that were stated earlier in the document. 
\begin{lemma}\label{lemma:total_bound}
Let Assumption \ref{ass:transverse} hold. Then there exists $C>0$ and $\epsilon_0>0$ such that for each $\xi \in \mc{X}$ and $t \in \sp{0,T}$ we have
 \begin{equation}
\norm{\oldphi_t(\xi)} \leq C
\end{equation} 
and for  each $\epsilon \leq \epsilon_0$ we have
\begin{equation}
\norm{\oldphi_t^\epsilon(\xi)} \leq C
\end{equation}
\end{lemma}
\begin{proof}
The proof follows closely the steps of \cite[Proposition 5.6.5]{polak2012optimization}. First, note that by the Lipshchitz continuity of $f_1$ and $f_2$ and the boundedness of $U$ there exists a constant $K> 0$ such that for each $x \in \R^n$ and $u \in U$ we have
\begin{equation}
\max\set{f_1(x,u), f_2(x,u)} \leq K (1 + \norm{x}).
\end{equation}
Furthermore we also see that:
\begin{equation}
\norm{f^\epsilon(x,u)} = \norm{\p{1-\phi\p{\frac{g(x)}{\epsilon}}}f_1(x,u) + \phi\p{\frac{g(x)}{\epsilon}}f_2(x,u)} \leq \max \set{\norm{f_1(x,u)}, \norm{f_2(x,u)}}  \leq K(1+ \norm{x})
\end{equation}
For each $\epsilon>0$ and $\xi  = (x_0, u) \in \mc{X}$ we also have
\begin{equation}
\norm{\oldphi_t^\epsilon(\xi)} \leq \norm{x_0} + \int_{0}^t \norm{f^\epsilon(\oldphi_t(\xi), u(t))} \leq K (1+ \oldphi_t^\epsilon(\xi)). 
\end{equation}
Using a standard Gronwall in equality \cite[Lemma 5.6.4]{polak2012optimization}, we conclude that 
\begin{equation}
\norm{\oldphi_t(\xi)} \leq e^K \p{1+ \norm{x_0}}.
\end{equation}
Recalling that $D$ is bounded and invoking Lemma \ref{lemma:convergence1} demonstrates the desired results. 
\end{proof}

\subsection{Proof of Lemma \ref{lemma:bounded}}
The following proof is rather involved, so we begin by making a number of simplifying assumptions for which we incur no loss of generality. To begin, we will assume that the nominal trajectory of the discontinuous system reaches the surface of discontinuity exactly once; the generalization to the case where it reaches the surface a finite number of times is straightforward. That is, we assume there is a single time $t_1 \in \mc{T}^\xi$. Without loss, we assume that $x_0 \in D_1$. To ease our analysis, we will assume that the function $g$ is such that for each $x = (x_1, \dots, x_n)^T \in \R^n$ we have $g(x) = x_1$. Note that by the Implicit Function Theorem there exists a set of coordinates about each point which makes this assumption true.

Furthermore, for claritrty of exposition, rather than prove the result using the adjoint equations, we will demonstrate it using the \emph{equations of sensitivity}. In particular, by the chain rule, we may also calculate $DL^\epsilon(\xi ; \delta \xi)$ as 
\begin{equation}
DL^\epsilon(\xi ; \delta \xi) = \deriv{}{x}L(\oldphi_T^\epsilon(\xi)) \cdot D\oldphi_T^\epsilon(\xi ; \delta \xi),
\end{equation}
where for each $t \in \sp{0,T}$ we define
\begin{equation}
D\oldphi_t^\epsilon(\xi ; \delta \xi) = \lim_{\lambda \downarrow 0}\frac{\oldphi_t^\epsilon(\xi + \lambda \delta \xi) - \oldphi_t^\epsilon(\xi)}{\lambda}
\end{equation}
where for each $t \in \sp{0,T}$ by \cite[Theorem 5.6.8]{polak2012optimization} we have
\begin{equation}\label{eq:sensitivity}
D\oldphi_t^\epsilon(\xi; \delta \xi) = \Phi(t,0) \delta x_0 + \int_{0}^{t} \Phi(t,\tau) \pderiv{}{u}f^\epsilon(\oldphi_t(\xi),u(\tau)) \delta u(\tau)\ \ \ \delta x(0) = \delta x_0,
\end{equation}
and, for each $\tau \in \sp{0,T}$, $\Phi(t, \tau)$ is the solution to
\begin{equation}
\deriv{}{t}\Phi(t,\tau) = \deriv{}{x}f^{\epsilon}(\oldphi_t^\epsilon(\xi),u(t)) \Phi(t,\tau), \ \ \ \Phi(\tau, \tau) = I,
\end{equation}
where in the above equation $I$ denotes the identity matrix. Note that since $\deriv{}{x}\oldphi_t(\xi)$

It is straightforward to check that, while $\pderiv{}{x}f^{\epsilon}(x,u)$  is of order $\frac{1}{\epsilon}$ in the region of regularization, $\pderiv{}{u}f^{\epsilon}(x,u)$ is bounded on bounded sets. Thus, if we can show that the state transition matrix in \eqref{eq:sensitivity} is bounded for each $t, \tau \in \sp{0,T}$, the desired result will follow immediately.

In order to demonstrate this fact, it is sufficient to demonstrate that the solution $\delta x^\epsilon \colon \sp{0,T}\to \R^n$ to
 \begin{equation}\label{eq:variation1}
 \deriv{}{t}\delta x^\epsilon(t) = \pderiv{}{x}f^\epsilon(\oldphi_t^\epsilon(\xi),u(t))\delta x^\epsilon(t), \ \ \ \delta x^\epsilon(0) =  \bar{x}_0
 \end{equation}
 remains bounded for each choice of initial condition $\bar{x}_0 \in \R^n$. Since we have assumed that the gradients of $f_1$ and $f_2$ are Lipschitz continuous, the solution to \eqref{eq:variation1} will remain bounded when the nominal regularized trajectory is away from the surface of discontinutiy; thus, we need only to show that the equations of sensitivity remain bounded when the nominal trajectory passes through the region of regularization.
 
 Throughout the proof, we will study the behavior of the sensitivity equations in three cases. We will first consider the case where he nominal trajectory of the discontinuous system crosses the surface of discontinuity transversally. In the second case, we will examine when the nominal trajectory of the discontinuous system begins to slide along the surface of discontinuity at time $t_1$. In the final case we consider, we will examine situations where the nominal trajectory of the discontinuous system either remains on the surface of discontinuity in an open interval containing $t_1$ without sliding, or instead instantaneously crosses the surface of discontinuity but does not do so transversally. 
 
 However, before addressing these various cases we first introduce notation and provide a preliminary analysis which will aid our discussion. We begin by defining for each $\epsilon>0$ and $t \in \sp{0,T}$
\begin{equation}
\phi^\epsilon(t) = \phi\p{\frac{g(\oldphi_t^\epsilon(\xi))}{\epsilon}}
\end{equation}
\begin{equation}
\omega^\epsilon(t) = \phi'\p{\frac{g(\oldphi_t^\epsilon(\xi))}{\epsilon}}
\end{equation}
\begin{equation}
A_1^\epsilon(t) = \pderiv{}{x}f_1(\oldphi_t^\epsilon(\xi),u(t))
\end{equation}
\begin{equation}
A_2^\epsilon(t) = \pderiv{}{x}f_2(\oldphi_t^\epsilon(\xi),u(t))
\end{equation}
\begin{equation}
\bar{f}^\epsilon(t) = (f_2(\oldphi_t^\epsilon(\xi),u(t))-f_1(\oldphi_t^\epsilon(\xi),u(t)) )
\end{equation}
We also define the following constant which will be used repeatedly in our analysis:
\begin{equation}
C_1 = \max_{a \in (-1, 1)} \pderiv{}{a} \phi(a)
\end{equation}
where we note that $C_1$ exists and is finite due to our assumptions on the partial derivatives of $\phi$. By our assumption that the nominal trajectory of the discontinuous system arrives at the surface of discontinuity transversally and Lemma \ref{lemma:convergence1} for each $\epsilon$ sufficiently small we may define
\begin{equation}
t^\epsilon = \inf\set{ t \in \sp{0,T} \colon \oldphi^\epsilon_t(\xi) \in \Sigma^\epsilon}
\end{equation}
to be the time instant at which the corresponding regularized trajectory first arrives at the region of regularization.

With the above notation we can compactly write the partial derivative
\begin{equation}
\pderiv{}{x}f^\epsilon(\oldphi_t(\xi),u(t)) = (1 - \phi^\epsilon(t))A_1^\epsilon(t) + \phi^\epsilon(t)A_2^\epsilon(t) + \frac{1}{\epsilon}\omega^\epsilon(t) \bar{f}^\epsilon(t)\nabla g(\oldphi_t^\epsilon(\xi))
\end{equation}

Applying Lemma \ref{lemma:total_bound}, and our assumption that the gradients of $f_1$ and $f_2$ are Lipzschitz, it is straightforward to bound for each $t \in \sp{0,T}$
\begin{equation}
\pderiv{}{x}f^\epsilon(\oldphi_t(\xi),u(t)) \leq M + M \omega^\epsilon(t)
\end{equation}
for some $M>0$, by noting that $\phi^\epsilon(t) \in \sp{0,1}$. 
 Applying a standard Gronwall inequality \cite[5.6.4]{polak2012optimization} we see that
\begin{equation}\label{eq:boundy1}
\norm{\delta x(t)}  \leq \exp\p{\int_{0}^t M + M\frac{1}{\epsilon}\omega^\epsilon(\tau) } \norm{\bar{x}} \leq  C_2 \exp\p{\int_{0}^t \frac{1}{\epsilon}M \omega^\epsilon(\tau) d \tau}\norm{\bar{x}}
\end{equation}
for some $C_2>0$ sufficiently large. 

This trivial bound allows us to simply bound the sensitivity of the regularized system in the case where the nominal trajectory of the discontinuous system crosses the surface of discontinuity transversally. Indeed, consider the case where there exists $\delta > 0 $ such that for almost every $t \in (t_1 - \delta , t_1 +\delta)$ we have $L_{f_1}g(\oldphi_t(\xi), u(t)) >0$ and $L_{f_2}g(\oldphi_t(\xi),u(t)) > 0$. Note that for $\epsilon$ sufficiently small we will also have $L_{f_1}(\oldphi^\epsilon_
t(\xi), u(t)) >0$ and $L_{f_2}(\oldphi^\epsilon_
t(\xi), u(t)) >0$ for each $t \in (t_1 -\delta, t_1+\delta)$.
Since the nominal trajectory of the discontinuous system passes transversally across the surface of discontinuity, it is straightforward to see that for each $\epsilon$ small enough the corresponding relaxed trajectory will spend at most on the order of $\epsilon$ units of time in the region of regularization. In other words, $\exists K >0$ such that $\oldphi_t^\epsilon(\xi) \notin \Sigma^\epsilon$ for each $t \not \in (t_1- K\epsilon, t_1 +k \epsilon)$. Combining this fact with \eqref{eq:boundy1}, the fact that $\omega^\epsilon(t) = 0$ if $\oldphi^\epsilon_t(\xi) \not \in \Sigma^\epsilon$ and the fact that $\omega^\epsilon(t)$ is bounded, it is straightforward to verify that in this case $\delta x^\epsilon(t)$ remains bounded for each $t \in \sp{0,T}$. Indeed we see that
\begin{equation}
\norm{\delta x(t)} \leq C_2 \exp\p{\int_{0}^t \frac{1}{\epsilon}M \omega^\epsilon(\tau) d \tau}\leq  C_2 \exp\p{\int_{t_1 -K\epsilon}^{t_1 + K \epsilon} M\frac{1}{\epsilon}C_1d \tau } \leq C
\end{equation}
for some $C>0$ sufficiently large, as desired.

In the second two cases that we consider, we will not be able to upperbound the time spent in the region of regularization on the order of $\frac{1}{\epsilon}$. Before proceeding to these more delicate conditions, we introduce some additional terminology which will clarify the intuition behind our approach to obtaining a bound in this case.

 First, we will further decompose
\begin{equation}\label{eq:diff_second}
\delta x^\epsilon(t) = \smat{\delta x_1^\epsilon(t) \\ \delta x_2^\epsilon(t)},
\end{equation}
where $\delta x_1^\epsilon \colon \sp{0,T} \to \R$ is the first entry of $\delta x^\epsilon$, and $\delta x^\epsilon_2 \colon \sp{0,T} \to \R^{n-1}$ represents the remaining n-1 entries. Given this notation, and our simplifying assumption on $g$ which implies $\nabla g(x) =\sp{1, 0, \dots , 0}$, we may rewrite \eqref{eq:variation1} as
\begin{equation}
\deriv{}{t} \delta x_1^\epsilon(t) =\p{ \bar{A}_{11}^\epsilon(t) + \frac{1}{\epsilon}\omega^\epsilon(t) \bar{f}_1^\epsilon(t)} \delta x_1^\epsilon(t) + \bar{A}_{12}^\epsilon(t) \delta x_2^\epsilon(t)
\end{equation}
and 
\begin{equation}\label{eq:diff_first}
\deriv{}{t}\delta x_2^\epsilon(t) = \p{ \bar{A}_{21}^\epsilon(t) +\frac{1}{\epsilon}\omega^\epsilon(t) \bar{f}_2^\epsilon(t)} \delta x_1^\epsilon(t) + \bar{A}_{22}^\epsilon(t)\delta x^\epsilon_2(t)
\end{equation}
where  $\bar{f}_1^\epsilon\colon \sp{0,T} \to \R$ denotes the first entry of $\bar{f}^\epsilon$ and $\bar{f}_2^\epsilon\colon \sp{0,T} \to \R^{n-1}$ denotes the $n-1$ remaining entries, and $\bar{A}_{11}^\epsilon \colon\sp{0,T} \to \R$, $\bar{A}_{12}^\epsilon \colon\sp{0,T} \to \R \times \R^{(n-1)}$, $\bar{A}_{21} \colon \sp{0,T} \to \R^{(n-1)} \times \R$ and $\bar{A}_{21}^\epsilon \colon \sp{0,T} \to \R^{(n-1)} \times \R^{(n-1)}$ are such that 
\begin{equation}
(1 -\phi^\epsilon(t))A_1^\epsilon(t)+ \phi^\epsilon(t)A_2^\epsilon(t) = \smats{\bar{A}_{11}^\epsilon(t) & \bar{A}_{12}^\epsilon(t) \\ \bar{A}_{21}^\epsilon(t) & \bar{A}_{22}^\epsilon(t)}.
\end{equation}
Solving these linear time-varying systems, we have that
\begin{multline}\label{eq:first}
\delta x_1^\epsilon(t) = \exp\p{\int_{0}^{t}  \bar{A}_{11}^\epsilon(\tau) + \frac{1}{\epsilon} \omega^\epsilon(\tau) \bar{f}_1^\epsilon(\tau) d\tau} \bar{x}_0^1 \\+ \int_{0}^{t} \exp\p{\int_{\tau}^{t} \bar{A}_{11}^\epsilon(s) +  \frac{1}{\epsilon} \omega^\epsilon(s) \bar{f}_1(s) ds} \bar{A}^\epsilon_{12}(\tau)\delta x_2^{\epsilon}(\tau)d \tau
\end{multline}
and 
\begin{equation}\label{eq:second}
\delta x_2^\epsilon(t) =  \exp\p{\int_{0}^{t} \frac{1}{\epsilon} \bar{A}_{22}^\epsilon(\tau) d\tau} \bar{x}_0^2 \\+ \int_{0}^{t} \exp\p{\int_{\tau}^{t}\bar{A}_{22}^\epsilon(s)ds} \p{\bar{A}_{12}^\epsilon(\tau) +\frac{1}{\epsilon}\omega^{\epsilon}(\tau) \bar{f}_2^\epsilon}\delta x_1^{\epsilon}(\tau) d\tau
\end{equation}
where $\bar{x}_0^1 \in \R$ is the first entry of $\bar{x}_0$, and $\bar{x}_0^2 \in \R^{(n-1)}$ represents the remaining $(n-1)$ entries. Here we immediately see the challenge of bounding the equations of sensitivity in this case. In particular, we observe that the solutions \eqref{eq:first} and \eqref{eq:second} will "blow up" unless $\delta x_1^\epsilon$ goes to zero rapidly upon entering the region of regularization. More accurately, we will need to show that $\delta x_1^\epsilon$ decays exponentially to zero at a rate that is on the order of $\frac{1}{\epsilon}$ to ensure that this does not occur.

We are now ready to consider the next case in our analysis. In this case, we assume that $\exists \delta >0$ such that for almost every $t \in (t_1 -\delta, t_2+\delta)$ we have $L_{f_1}g(\oldphi_t(\xi), u(t))>0$ and $L_{f_2}g(\oldphi_t(\xi), u(t))<0$. In this case the nominal trajectory of the discontinuous system begins to slide along the surface of discontinuity at time $t_1$. Crucially, we note that in this case for each $t \in (t_1 -\delta , t_1 +\delta)$ we will have
\begin{equation}
\bar{f}_1^\epsilon(t) < - \hat{f}
\end{equation}
for some $\hat{f} > 0$ and each $\epsilon$ sufficiently small, since $f_1$ and $f_2$ are both "pointing into" the surface of regularization in this case. By our assumption that $L_{f_1}g(\oldphi_t(\xi), u(t)) >0$ and $L_{f_2}g(\oldphi_t(\xi),u(t)) <0$, and our assumption that $\phi$ is strictly increasing on $(-1,1)$, we see that when $\epsilon$ is sufficiently small that the regularized trajectory will flow to the relative interior of $\Sigma^\epsilon$ after arriving at the region of regularization at time $t^\epsilon$, and remain in the relative interior of $\Sigma^\epsilon$ until at least time $t_1 +\delta$. Moreover, since $\deriv{}{x}\phi(\frac{g(x)}{\epsilon}) >0$ for each $x \in \R^n$ which lies in the relative interior of $\Sigma^\epsilon$, we see that there exists $k_1>0$ such that  $\omega^\epsilon(t) > k_1$ for each $t \in (\bar{t}^\epsilon, t_1 + \delta)$ where we define 
\begin{equation}
\bar{t}^\epsilon = t^\epsilon +\epsilon.
\end{equation}
Thus, the differential equation \eqref{eq:diff_first} will be damped on the order of $\frac{1}{\epsilon}$ when the regularized trajectory is in the region of regularization. However, the differential equation \eqref{eq:diff_second} will will simultaneously "blow up" exponentially at a rate that is on the order of $\frac{1}{\epsilon}$. Due to the interconnection of these two differential equations, it is not imediatly clear which of these effects will prevail. Nonetheless, we will still show that for $\epsilon$ sufficiently small $\delta x_1^\epsilon$ still decays to zero exponentially at a rate that is on the order of $\frac{1}{\epsilon}$. Once $\delta x^\epsilon_2(t)$ decays to the point that it is $(O(\epsilon))$, the right hand sides of \eqref{eq:diff_first} and \eqref{eq:diff_second} will $O(1)$ for the rest of the time that the regularized trajectory remains in the region of regularization, and we will have obtained the bound that we desire.

Before proceding to our main analytic step, we first develop a number of useful intermediate bounds. By lemma \ref{lemma:total_bound} we see that $\bar{A}_{ij}^\epsilon(t)$ is bounded for each $t \in \sp{0,T}$ and $i , j \in \set{1,2}$, and $\bar{f}^\epsilon_2(t)$ is also bounded for each $t \in \sp{0,T}$, and and thus using familiar  properties of the exponential function we bound for each $t \in \sp{0,T}$
\begin{multline}\label{eq:boundy2}
\norm{\delta x_1^\epsilon(t)} \leq \exp\p{\int_0^{t} \bar{A}_{11}^\epsilon(\tau) d \tau } \exp\p{\int_0^t \frac{1}{\epsilon} \omega^\epsilon(\tau)\bar{f}_1^\epsilon(\tau) d \tau} \norm{\bar{x}_0} \\+ \int_{0}^t\exp\p{\int_\tau^{t} \bar{A}_{11}^\epsilon(s) d s } \exp\p{\int_\tau^t \frac{1}{\epsilon} \omega^\epsilon(s)\bar{f}_1^\epsilon(s) d s} \norm{\delta x_2^{\epsilon}(\tau)} d \tau \\
\leq C_3 \exp\p{\int_{0}^{t}  \frac{1}{\epsilon} \omega^\epsilon(\tau) \bar{f}_1(\tau) d\tau} \norm{\bar{x}_0} +  C_3 \int_{0}^{t} \exp\p{\int_{\tau}^{t}  \frac{1}{\epsilon} \omega^\epsilon(s) \bar{f}_1(s) ds}\norm{\delta x_2^{\epsilon}(\tau)}d \tau
\end{multline}
and
\begin{multline}\label{eq:boundy3}
\norm{\delta x_2^\epsilon(t)} \leq  \exp\p{ \int_0^t \bar{A}_{22}^\epsilon(\tau) d \tau}\norm{\bar{x}_0} + \int_0^t \exp\p{ \int_\tau^t \bar{A}_{22}^\epsilon(s) d s}\frac{1}{\epsilon}\omega^\epsilon(\tau)\bar{f}_2^\epsilon(\tau) \\
 \leq 
C_3 \norm{\bar{x}_0} + C_3\int_0^t \frac{1}\epsilon \omega^{\epsilon}(\tau) \norm{\delta x^\epsilon_1(\tau)} d \tau
\end{multline}
where we chose $C_3$ to be sufficiently large and we note that $\norm{\bar{x}_0^1}, \norm{\bar{x}_0^2} \leq \norm{\bar{x}_0}$ by the Triangle Inequality. 

Using these new bounds, for each $t \in (\bar{t}^\epsilon, t_1 + \delta)$ we are employ \eqref{eq:boundy2}  and \eqref{eq:boundy1} and our definitions of $k_1$ and $\hat{f}$ we further decompose and bound 
\begin{multline}\label{eq:boundy5}
\norm{\delta x_1^\epsilon(t)} \leq C_3  \exp\p{\int_{t^\epsilon}^{\bar{t}^\epsilon}  \frac{1}{\epsilon} \omega^\epsilon(\tau) \bar{f}_1(\tau) d\tau} \exp\p{\int_{\bar{t}^\epsilon}^{t}  -\frac{1}{\epsilon}   k_1 \hat{f}  d\tau} \norm{\bar{x}_0} \\
 +  C_3 \int_{t^\epsilon}^{\bar{t}^\epsilon} \exp\p{\int_{\tau}^{t}  \frac{1}{\epsilon} \omega^\epsilon(s) \bar{f}_1(s) ds}\norm{\delta x_2^{\epsilon}(\tau)}d \tau + C_3\int_{\bar{t}^\epsilon}^{t} \exp\p{\int_{\tau}^{t}  -\frac{1}{\epsilon}   k_1 \hat{f} ds}\norm{\delta x_2^{\epsilon}(\tau)}d \tau \\
 \leq C_3  \exp\p{\int_{t^\epsilon}^{\bar{t}^\epsilon}  \frac{1}{\epsilon} \omega^\epsilon(\tau) \bar{f}_1(\tau) d\tau} \exp\p{\int_{\bar{t}^\epsilon}^{t}  -\frac{1}{\epsilon}   k_1 \hat{f}  d\tau} \norm{\bar{x}_0} \\
 + C_3 \int_{t^\epsilon}^{\bar{t}^\epsilon} \exp\p{\int_{\tau}^{t}  \frac{1}{\epsilon} \omega^\epsilon(s) \bar{f}_1(s) ds}C_2\exp\p{\int_0^t \frac{1}{\epsilon} M \omega^\epsilon(s)ds}\norm{\bar{x}_0}d \tau + C_3\int_{\bar{t}^\epsilon}^{t} \exp\p{\int_{\tau}^{t}  -\frac{1}{\epsilon}   k_1 \hat{f} ds}\norm{\delta x_2^{\epsilon}(\tau)}d \tau\\
\leq C_4 \exp\p{\int_{\bar{t}^\epsilon}^{t} -\frac{1}{\epsilon}\hat{f}k_1 d\tau} \norm{\bar{x}_0} +  C_4\int_{\bar{t}^\epsilon}^{t} \exp\p{\int_{\tau}^{t}  -\frac{1}{\epsilon}  \hat{f}k_1 ds}\norm{\delta x_2^{\epsilon}(\tau)}d \tau
\end{multline}
for some $C_4>0$ sufficiently large, where we have use \eqref{eq:boundy1}. 

We can follow a similar set of steps and employ  \eqref{eq:boundy1} to obtain:
\begin{multline}\label{eq:boundy6}
\norm{\delta x_2^\epsilon(t)} \leq C_2 \exp\p{\int_{0}^t \frac{1}{\epsilon}M \omega^\epsilon(\tau) d \tau}\norm{\bar{x}_0}\\
 \leq C_2 \exp\p{\int_{\hat{t}^\epsilon}^{\hat{t}^\epsilon} \frac{1}{\epsilon}M \omega^\epsilon(\tau) d \tau} \exp\p{\int_{\hat{t}^\epsilon}^t \frac{1}{\epsilon}M \omega^\epsilon(\tau) d \tau}\norm{\bar{x_0}}\\ 
\leq C_5 \exp\p{\int_{\hat{t}^\epsilon}^t \frac{1}{\epsilon}M C_2 d \tau}\norm{\bar{x_0}}
\end{multline}
for some $C_5 >0$ sufficiently large.

Here, we pause to inspect the last inequality in \eqref{eq:boundy5}. In order to show that $\delta x_1^\epsilon$ decays exponentially on the order of $\frac{1}{\epsilon}$, we will need to bound $\norm{\delta x_1^\epsilon}$ on an interval of appropriate length. In particular, observe that if we can show that $\norm{\delta x^{\epsilon}(t)} < C$ for some $C>0$ for each $t \in (\bar{t}^\epsilon, \bar{t}^\epsilon + \frac{1}{k_2*\hat{f}}*\epsilon \ln(\frac{1}{\epsilon}))$, then $\delta_1^\epsilon(t)$ will reach a value of $O(\epsilon)$ by then end of this interval. Once this occurs, both $\delta x^\epsilon_1$ and $\delta x_2^\epsilon$ will vary at a rate that is on the order of $O(1)$ for the rest of the time that the regularized trajectory remains in the region of regularization, and we will have obtained the bound we desire. Though we will prove this more carefully later, we provide this commentary to provide an intuition for our following steps, which my be otherwise difficult to follow. 

In order to bound $\delta x^\epsilon_2$ uniformly on the desired interval, we will use the following procedure to repeatedly tighten bounds that we have already obtained for $\delta x_1^\epsilon$ and $\delta x_2^\epsilon$ in an inductive manner. In particular, we will begin by plugging in \eqref{eq:boundy6} into \eqref{eq:boundy5} and then taking several additional steps. The result will be a tighter version of \eqref{eq:boundy6}. We will then plug this tight bound into \eqref{eq:boundy5} and and again repeat the intermediate steps until we have a uniform bound for $\norm{\delta x_1^\epsilon(t)}$ on our desired interval.

To begin, we plug \eqref{eq:boundy6} into \eqref{eq:boundy5}, and integrate to obtain:
\begin{multline}\label{eq:step_1}
\norm{\delta x_1^\epsilon(t)} \leq C_4 \exp\p{-\frac{1}{\epsilon}\hat{f}k_1 (t- \bar{t}^\epsilon)}||\bar{x}_0|| +C_4 C_5\int_{\bar{t}^\epsilon}^{t} \exp\p{-\frac{1}{\epsilon}k_1 \hat{f}(t -\tau)}\exp\p{\frac{1}{\epsilon}C_2 M(\tau -\bar{t}^{\epsilon})} \norm{x_0} d \tau \\ 
= C_4 \exp\p{-\frac{1}{\epsilon}\hat{f}k_1 (t- \bar{t}^\epsilon)}||\bar{x}_0||  + C_4C_5\frac{\epsilon}{M*C_2 + k_1*\hat{f}} \p{\exp\p{\frac{1}{\epsilon}C_2*M (t -\bar{t}^\epsilon)} - 1 )}\norm{\bar{x}_0}.
\end{multline}
Next we plug this expression into \eqref{eq:boundy3} and again use $\omega^\epsilon(t) \leq C_2$ to obtain
\begin{multline}
\norm{\delta x_2^\epsilon(t)} \leq C_3 \norm{\bar{x}_0} +\norm{\bar{x}_0} C_3\int_{\bar{t}^\epsilon}^t \frac{1}\epsilon C_4\bigg[ \exp\p{-\frac{1}{\epsilon}\hat{f}k_1 (t- \bar{t}^\epsilon)} \\ + C_5\frac{\epsilon}{M*C_2 + k_1*\hat{f}} \p{\exp\p{\frac{1}{\epsilon}C_2*M (t -\bar{t}^\epsilon)} - 1 )} \bigg] d \tau
\end{multline}
and then pulling out constants we then obtain
\begin{multline}
\norm{\delta x_2^\epsilon(t)} \leq C_6 \norm{\bar{x}_0 }\bigg[ 1 + \int_{\bar{t}^\epsilon}^t \frac{1}\epsilon\bigg[ \exp\p{-\frac{1}{\epsilon}\hat{f}k_1 (t- \bar{t}^\epsilon)}  +\epsilon \exp\p{\frac{1}{\epsilon}C_2*M (t -\bar{t}^\epsilon)}  \bigg] d \tau \bigg]
\end{multline}
for some $C_6>0$ sufficiently large, and then we integrate to obtain
\begin{equation}
\norm{\delta x_2^\epsilon(t)} \leq C_6 \norm{\bar{x}_0 }\bigg[ 1 + \frac{1}{\hat{f}*k_1}\bigg[ \exp\p{-\frac{1}{\epsilon}\hat{f}k_1 (t- \bar{t}^\epsilon)} -1 \bigg]  +\frac{\epsilon}{C_2*M}\bigg[ \exp\p{\frac{1}{\epsilon}C_2*M (t -\bar{t}^\epsilon)}  -1 \bigg] \bigg].
\end{equation}
Again pulling out constants, and noting that for each appropriate choice of $t$ and $\tau$ we have
\begin{equation}
\exp\p{-\frac{1}{\epsilon}\hat{f}k_1 (t- \bar{t}^\epsilon)}\leq 1
\end{equation}
we then obtain
\begin{equation}\label{eq:ep_bound1}
\norm{\delta x_2^\epsilon(t)} \leq C_7 \norm{\bar{x}_0} \p{1 + \epsilon \exp\p{\frac{1}{\epsilon}C_2*M (t -\bar{t}^\epsilon)}}.
\end{equation}
for some $C_7>0$ sufficiently large.
On the desired interval $t \in (\bar{t}^\epsilon, \bar{t}^\epsilon + \frac{1}{k_2*\hat{f}}*\epsilon \ln(\frac{1}{\epsilon}))$ we have that
\begin{multline}
\norm{\delta x_2^\epsilon(t)} \leq C_7 \norm{\bar{x}_0} \p{1 + \epsilon \exp\p{\frac{C_2*M}{k_2*\hat{f}} \ln(\frac{1}{\epsilon})}} = C_7 \norm{\bar{x}_0}\p{1 + \epsilon *\epsilon^{-\frac{C_2*M}{k_2*\hat{f}}}} = C_7 \norm{\bar{x}_0}\p{1 + \epsilon^{1-\frac{C_2*M}{k_2*\hat{f}}}}
\end{multline}
In the case that $\frac{C_2*M}{k_2*\hat{f}}<1$ we will have that $\delta x^\epsilon_2$ is bounded on the desired interval. However, in the case that   $\frac{C_2*M}{k_2*\hat{f}}>1$ the above bound will tend to infinity as we take $\epsilon \to 0$ and will not serve our needs. However, we can now plug $\eqref{eq:ep_bound1}$ into our above analysis to tighten the bound. In particular, rather than plug \eqref{eq:boundy6} into \eqref{eq:boundy5} in equation \eqref{eq:step_1}, we can now put in the improved bound \eqref{eq:ep_bound1}, and follow the steps take from \eqref{eq:step_1} to \eqref{eq:ep_bound1} to improve the bound. Though we omit the details in the interest of brevity, it can be shown that this proccess yields the following bound:
\begin{equation}
\norm{\delta x_2^\epsilon(t)} \leq C_8 \norm{\bar{x}_0} \p{1 + \epsilon^2 \exp\p{\frac{1}{\epsilon}C_2*M (t -\bar{t}^\epsilon)}},
\end{equation}
for some $C_8>0$ which is chosen to be sufficiently large. If we instead repeat the above process $p$ times, where $p> \frac{C_2*M}{k_2*\hat{f}}$, then we obtain abound of the form
\begin{equation}
\norm{\delta x_2^\epsilon(t)} \leq C_p \norm{\bar{x}_0} \p{1 + \epsilon^p \exp\p{\frac{1}{\epsilon}C_2*M (t -\bar{t}^\epsilon)}}
\end{equation}
for some $C_p>0$ sufficiently large, and in this case we may bound $\delta x_2^\epsilon(t)$ for each $t \in (\bar{t}^\epsilon, \bar{t}^\epsilon + \frac{1}{k_2*\hat{f}}*\epsilon \ln(\frac{1}{\epsilon}))$ by
\begin{multline}
\norm{\delta x_2^\epsilon(t)} \leq C_p \norm{\bar{x}_0} \p{1 + \epsilon^p \exp\p{\frac{C_2*M}{k_2*\hat{f}} \ln(\frac{1}{\epsilon})}} = C_p \norm{\bar{x}_0}\p{1 + \epsilon^p *\epsilon^{-\frac{C_2*M}{k_2*\hat{f}}}} = C_p \norm{\bar{x}_0}\p{1 + \epsilon^{p-\frac{C_2*M}{k_2*\hat{f}}}},
\end{multline}
and in this case we will have that for each $t \in (\bar{t}^\epsilon, \bar{t}^\epsilon + \frac{1}{k_2*\hat{f}}*\epsilon \ln(\frac{1}{\epsilon}))$
\begin{equation}
\norm{\delta x_2^\epsilon(t)} \leq C_9\norm{\bar{x}_0}, 
\end{equation}
where $C_9>0$ is chosen to be sufficiently large.
Combining this bound with \eqref{eq:boundy5}, on the interval  $t \in (\bar{t}^\epsilon, \bar{t}^\epsilon + \frac{1}{k_2*\hat{f}}*\epsilon \ln(\frac{1}{\epsilon}))$ we have
\begin{multline}\label{eq:exp}
\norm{\delta x_1^\epsilon(t)} \leq C_4 \exp\p{\int_{\bar{t}^\epsilon}^{t} -\frac{1}{\epsilon}\hat{f}k_1 d\tau} \norm{\bar{x}_0} +  C_4\int_{\bar{t}^\epsilon}^{t} \exp\p{\int_{\tau}^{t}  -\frac{1}{\epsilon}  \hat{f}k_1 ds}C_9 \norm{\bar{x}_0}d \tau \\
\leq C_{11} \exp\p{\int_{\bar{t}^\epsilon}^{t} -\frac{1}{\epsilon}\hat{f}k_1 d\tau} \norm{\bar{x}_0}  +  \epsilon C_{11} \norm{\bar{x}_0}
\end{multline}
for some $C_{11} >0$ sufficiently large and setting  $\hat{t}^\epsilon =\bar{t}^\epsilon + \frac{1}{k_2*\hat{f}}*\epsilon \ln(\frac{1}{\epsilon})$ we have
\begin{equation}
\norm{\delta_1x^\epsilon(\hat{t}^\epsilon)} \leq C_4 \epsilon \norm{\bar{x}_0} + C_9 *\frac{\epsilon}{\hat{f}k_1 }.
\end{equation}
Thus, we see that $\norm{\delta x^\epsilon}$ decays exponentially to zero at a rate of decay that is on the order of $\frac{1}{\epsilon}$, as desired. Moreover, combing \eqref{eq:exp} with \eqref{eq:exp} and integrating through we obtain for each $t \in (\bar{t}^\epsilon, \bar{t}^\epsilon + \frac{1}{k_2*\hat{f}}*\epsilon \ln(\frac{1}{\epsilon}))$
\begin{equation}
\norm{\delta x_2^\epsilon(t)} \leq C_3 \norm{\bar{x}_0} + C_3\int_0^t \frac{1}\epsilon C_2 \p{C_{11} \exp\p{\int_{\bar{t}^\epsilon}^{t} -\frac{1}{\epsilon}\hat{f}k_1 d s} \norm{\bar{x}_0}  + \epsilon C_{11}}d \tau\\
\leq C_{12} \norm{\bar{x}_0}
\end{equation} 
for some $C_{12}>0$ sufficiently large. From here it is straightforward to verify that the equations of sensitivity remain bounded for the rest of the time that the nominal trajectory is in the region of regularization. We omit a detailed argument in the interest of brevity, since it closely follows the argument made in \cite{stewart2010optimal}.

Finally, we discuss the case where the nominal trajectory of the discontinuous system either stays on the surface of discontinuity with out sliding, and the case where the nominal trajectory of the discontinuous system instantaneously crosses the surface of discontinuity, but does not do so transversally. This covers the remaining situations we must consider under Assumptions \ref{ass:diff1}-\ref{ass:diff5}. We provide only a sketch of the proofs in these case, since the arguments follow largely from our previous analysis. 

First, suppose that the nominal trajectory of the discontinuous system stays on the surface of discontinuity on the interval $(t_1, t_2)$ where $t_2 >0$ but does not slide. For almost every time $t_\in (t_1,t_2)$, we will have either $L_{f_1}g(\oldphi_t(\xi),u(t)) >0$ and $L_{f_1}g(\oldphi_t(\xi),u(t)) \leq 0$ or  $L_{f_1}g(\oldphi_t(\xi),u(t)) \geq 0$ and $L_{f_1}g(\oldphi_t(\xi),u(t)) < 0$, otherwise the trajectory would leave the surface of discontinuity before time $t_2$. Since either $f_1$ or $f_2$ is always "pointing into" the surface of discontinuity, for each $\epsilon$ small enough, we will again have $\bar{f}_1^\epsilon(t) <0$ for almost every time $t \in (t_1,t_2)$, and we will again have exponential dampening of variations that are normal to the surface of discontinuity. Thus, the analysis to show that the equations of sensitivity are bounded in this case follows closely the previous case with a few minor modifications. 

In the case where the nominal trajectory crosses the surface of discontinuity but does not do so transversally, it will be the case that there exists $\delta $ such that for almost every $t \in (t_1 -\delta ,  t_1 + \delta )$ we have $L_{f_1} g(\oldphi_t(\xi),u(t))> 0$ and $L_{f_2}g(\oldphi_t(\xi),u(t))\geq 0$. Moreover, given any $\gamma>0$ we may choose $\delta$ to be to be small enough so that $ \gamma \leq L_{f_2}g(\oldphi_t(\xi),u(t))$ for almost every $t \in (t_1 -\delta, t_1+\delta)$. This is a consequence of the trajectory not leaving the surface of discontinuity transversally. Thus, for $\epsilon$  small enough, will again have that $\bar{f}_1^\epsilon(t) <0$ for almost every $t$ such that $\sigma_t^\epsilon(\xi) \in \Sigma^\epsilon$. Thus, the previous analysis can again can be applied to this situation with minor modifications to show that the variations on the regularized trajectory remain bounded as the trajectory flows through the interior of the region of regularization. We omit the details of the proof for these final two cases, because it would involve reintroducing the large number of bounds and notations introduced in the case of sliding without providing much new insight.

\subsection{Proof of Lemma \ref{lemma:neighborhood}}
By the proof of Lemma \ref{lemma:bounded} we see that that there exists a neighborhood of $\xi$ on which the map $\oldphi_t$ is continuous for each $t \in \sp{0,T}$ (this is a well known property of Filippov solutions \cite[Chapter 2, Section 7]{filippov2013differential}). This demonstrates assumptions \ref{ass:diff1} and \ref{ass:diff5} hold in a neighborhood of $\xi$. We then see that \ref{ass:diff3} holds in a neighborhood of $\xi$ by noting that $f_1$ and $f_2$ are assumed to be continuous.

\subsection{Proof of Lemma \ref{lemma:Lipschitz}}
In order to again simplify our analysis, we again make the simplifying assumptions made in the proof of Lemma \ref{lemma:bounded}. That is, we again assume that the nominal trajectory of the discontinuous system reaches the surface of discontinuity exactly once at time $t_1 \in \mc{T}^\xi$ and that $x_0 \in D_1$. We also reuse all of the terminology and notation from the proof of Lemma \ref{lemma:bounded}. 
We again prove the desired result by studying the equations of sensitivity for the relaxed system, and demonstrating that the map $\epsilon \to D\oldphi_t(\xi ; \delta \xi)$ is Lipschitz continuous when we restrict our attention to values of $\epsilon$ that are small enough. 

For each $\epsilon>0$, we define the \emph{fast time scale} 
\begin{equation}
\tau = \frac{t- t^\epsilon}{\epsilon}
\end{equation}
on this timescale, the the dynamics of the sensitivity equations are given by
\begin{equation}
\deriv{}{\tau}\delta x^{\epsilon}(\tau) = \epsilon \sp{\p{1- \phi^{\epsilon}(\tau)}A_{11}^{\epsilon}(\tau) + \phi^\epsilon(\tau)}\delta x^\epsilon(\tau) + \omega^{\epsilon}(\tau)f_1(\tau) \delta x_1^{\epsilon}(\tau).
\end{equation}
We again begin by studying the case where the nominal trajectory of the discontinuous systems crosses the surface of discontinuity transversally. That is, we assume that there exists $\delta >0$ such that for each  $t \in (t_1-\delta , t_1 +\delta)$ we have $L_{f_1}g(\oldphi_t(\xi), u(t)) >0$ and $L_{f_2}g(\oldphi_t(\xi),u(t))>0$. Similarly, to the proof of Lemma \ref{lemma:bounded}, our analysis in this case largely relies on the fact that for each $\epsilon$ small enough the corresponding regularized trajectory will spend on the order of $\epsilon$ units of time in the region of regularization. In particular, we begin by bounding:

\newpage a \newpage

\bibliography{paper}

\end{document}